\title{Divergence of CAT(0) Cube Complexes and \\Coxeter Groups}
\author{Ivan Levcovitz}
\date{}

\documentclass[12pt]{article}

\usepackage[]{hyperref}

\usepackage{amsmath}
\usepackage{amssymb}
\usepackage{amsthm}

\usepackage[normalem]{ulem}
\usepackage{graphicx}
\usepackage{overpic}
\usepackage{enumitem}

\newtheorem{theorem}{Theorem}[section]
\newtheorem{lemma}[theorem]{Lemma}

\newtheorem{corollary}{Corollary}[theorem]

\newtheorem{subclaim}[theorem]{Sub-Claim}
\theoremstyle{definition}
\newtheorem{definition}[theorem]{Definition}

\theoremstyle{remark}
\newtheorem{remark}{Remark}[theorem]

\long\def\Restate#1#2#3#4{
\medskip\par\noindent
{\bf #1 \ref{#2} #3} {\it #4}\par\medskip }

\long\def\Doublestate#1#2#3#4{
\medskip\par\noindent
{\bf #1 \ref{#2} and \ref{#3}} {\it #4}\par\medskip }

\usepackage[bottom=1in, left=1in, right=1in, top=1in]{geometry}

\usepackage{sectsty}

\sectionfont{\fontsize{12}{15}\selectfont \centering}

\subsectionfont{\fontsize{12}{15}\selectfont }

\begin{document}
\maketitle

\abstract{We provide geometric conditions on a pair of hyperplanes of a CAT(0) cube complex that imply divergence bounds for the cube complex. As an application, we classify all right-angled Coxeter groups with quadratic divergence and show right-angled Coxeter groups cannot exhibit a divergence function between quadratic and cubic. This generalizes a theorem of Dani-Thomas that addressed the class of $2$--dimensional right-angled Coxeter groups. As another application, we provide an inductive graph theoretic criterion on a right-angled Coxeter group's defining graph which allows us to recognize arbitrary integer degree polynomial divergence for many infinite classes of right-angled Coxeter groups. We also provide similar divergence results for some classes of Coxeter groups which are not right-angled.}

\section{Introduction}
In this article we study the divergence of CAT(0) cube complexes and apply our results to the class of right-angled Coxeter groups. We also provide results regarding the divergence of Coxeter groups which are not right-angled. 

Given a metric space $X$ and a positive number $r$, the divergence function $Div(X, r) = Div(X)$ is the supremum over all lengths of minimal paths, which avoid a ball of radius $r$, connecting two points that are distance roughly $r$ apart. One may roughly think of the divergence function as a measure of the best upper bound on the rate a pair of geodesic rays can stray apart from one another. For a finitely generated group $G$, $Div(G)$ is the divergence function applied to the Cayley graph of $G$ endowed with the word metric.

Groups which are $\delta$--hyperbolic all exhibit at least exponential divergence. On the other end of the spectrum, $\mathbb{Z}^n$ displays linear divergence for $n \ge 2$. Gromov conjectured that spaces of nonpositive curvature should exhibit either linear or at least exponential divergence \cite{Gro2}. To the contrary, many important classes of groups, several of which are CAT(0), contain groups of quadratic divergence. For instance, the class of $3$--manifold groups \cite{Ger, KL}, the mapping class group of a closed surface of genus $g \ge 2$ \cite{Beh}, right-angled Artin groups \cite{BC} and right-angled Coxeter groups \cite{DT} have been shown to contain groups of quadratic divergence.  

More recently, there have been some examples of CAT(0) groups exhibiting polynomial divergence of any positive integer degree \cite{BD} \cite{BeHa} \cite{DT} \cite{Mac}. Additionally, there are even constructions of more exotic infinitely presented groups (not necessarily CAT(0)) with divergence function not a polynomial \cite{GS} \cite{OOS}. 

Right-angled Artin groups have played a central role in contemporary mathematics (see for example \cite{Wise} \cite{Ago}). In terms of their divergence, these groups satisfy a certain trichotomy: each right-angled Artin group either exhibits linear, quadratic or infinite divergence with these occurrences classified by simple properties of the group's defining graph \cite{BC}. Fundamental groups of $3$--manifolds exhibit a similar trichotomy as well \cite{Ger2}.  

Right-angled Coxeter groups form an important class of groups which act geometrically on CAT(0) cube complexes. Associated to any simplicial graph $\Gamma$ is a right-angled Coxeter group, $W_{\Gamma}$, whose presentation consists of an order 2 generator for each vertex of $\Gamma$ with the relation that two generators commute if there is an edge between the corresponding vertices of $\Gamma$.  It is true that every right-angled Artin group is finite index in some right-angled Coxeter group \cite{DJ}, therefore as divergence is a quasi-isometry invariant, the class of right-angled Coxeter groups contains groups of linear, quadratic and infinite divergence as well. However, even more is true for these groups.

For any positive integer degree, Dani-Thomas surprisingly provide an example of a $2$--dimensional right-angled Coxeter group exhibiting polynomial divergence of the given degree \cite{DT}. This raises the question of which divergence functions are possible for right-angled Coxeter groups. Additionally, there is the question of which properties of right-angled Coxeter groups give rise to their broader spectrum of divergence functions and how can these properties be recognized through these groups' defining graphs. 

The class of groups that act geometrically on a CAT(0) cube complex is vast and includes both right-angled Artin groups and right-angled Coxeter groups. More generally, we ask which properties of CAT(0) cube complexes give rise to different divergence functions. 

The Rank Rigidity Theorem shows the existence of a rank one isometry in the automorphism group of an irreducible, essential, locally compact CAT(0) cube complex with cocompact automorphism group \cite{CS} (the result actually holds under more general assumptions as well). As a consequence, the divergence of these spaces is either linear or at least quadratic \cite{Hag2}. 

We prove that for the case of right-angled Coxeter groups, there is an additional gap between quadratic and cubic divergence, and we classify exactly which right-angled Coxeter groups exhibit quadratic divergence. 

\Restate{Theorem}{racg_quadratic_div_theorem}{}
{
Suppose the graph $\Gamma$ is not a non-trivial join. The right-angled Coxeter group $W_{\Gamma}$ exhibits quadratic divergence if and only if $\Gamma$ is CFS. If $\Gamma$ is not CFS, then the divergence of $W_{\Gamma}$ is at least cubic. 
}

The CFS condition (``constructed from squares'') is a purely graph-theoretic condition which can be computationally checked. We note that the divergence of $W_{\Gamma}$ is linear if and only if $\Gamma$ is a non-trivial join \cite{BFHS}. For the case when $\Gamma$ does not contain triangles, the above theorem is a result of Dani-Thomas \cite{DT}. Such groups are precisely those whose Davis complex, a natural CAT(0) space a Coxeter group acts on, is $2$--dimensional. Theorem \ref{racg_quadratic_div_theorem} thus generalizes Dani-Thomas's result to right-angled Coxeter groups of arbitrary dimension. 
 
An important application of Theorem \ref{racg_quadratic_div_theorem} is to the theory of random right-angled Coxeter groups. Let $\Gamma(n, p(n))$ be a random $n$-vertex graph containing an edge between a given pair of vertices with probability $p(n)$. A random right-angled Coxeter group is simply the right-angled Coxeter group defined by a random graph. Behrstock--Falgas-Ravry--Hagen--Susse \cite{BFHS} give a threshold theorem for when a random graph is CFS with probability 1. Combining their result with Theorem \ref{racg_quadratic_div_theorem}, we obtain a threshold function for the transition between quadratic to at least cubic divergence in random right-angled Coxeter groups. 

\begin{theorem}[Behrstock--Falgas-Ravry--Hagen--Susse, Levcovitz] \label{random_racg_theorem}
	Suppose $p(n)$ is a probability density function bounded away from 1 and let $\epsilon > 0$. Let $\Gamma = \Gamma(p(n), n)$ be a random graph. If $p(n) > n^{-\frac{1}{2} + \epsilon}$, then the right-angled Coxeter group $W_{\Gamma}$ asymptotically almost surely exhibits quadratic divergence. If $p(n) < n^{-\frac{1}{2} - \epsilon}$, then $W_{\Gamma}$ asymptotically almost surely exhibits at least cubic divergence. 
\end{theorem} 

We note that a random right-angled Coxeter group given as above asymptotically almost surely has dimension larger than two. Thus the generality of Theorem \ref{racg_quadratic_div_theorem} is needed to obtain Theorem \ref{random_racg_theorem}.  

\textit{Strongly thick metric spaces of order $d$} form an important class of spaces which can be constructed through a $d$--step inductive gluing procedure, with initial pieces of linear divergence. An important consequence is that these spaces must have divergence bounded above by a polynomial of degree $d+1$ \cite{BD}. There are not many general results in the opposite direction giving lower bounds on divergence, and a goal of this article is to introduce criteria which imply such lower bounds. These criteria then allow us to give the exact divergence, up to an equivalence of functions, for many spaces. 

We apply the following strategy to study the divergence in CAT(0) cube complexes. First, we define the \textit{hyperplane divergence function}, $HDiv$, that given a pair of non-intersecting hyperplanes, gives the length of a shortest path between these hyperplanes that avoids a ball of radius $r$ about a basepoint. We then give conditions on a pair of non-intersecting hyperplanes that imply a lower bound on their corresponding hyperplane divergence function. The proofs for these lower bounds involve the use of disk diagrams. Finally, we show how the hyperplane divergence function for a pair of such hyperplanes actually implies a lower bound on the divergence of the entire CAT(0) cube complex: 

\Restate{Theorem}{chaining_hyperplanes_theorem}{} 
{
Let $X$ be an essential, locally compact CAT(0) cube complex with cocompact automorphism group. Suppose $HDiv(\mathcal{Y}, \mathcal{Z}) \succeq F(r)$ for a pair of non-intersecting hyperplanes $\mathcal{Y}$ and $\mathcal{Z}$ in $X$. It then follows that $Div(X) \succeq rF(r)$. 
}

Consequently, this process reduces the problem of finding a lower bound on divergence to finding a pair of hyperplanes with certain separation properties. Through this strategy, we prove the following theorem which gives lower bounds on divergence as a consequence of the existence of certain types of pairs of non-intersecting hyperplanes (these hypotheses on hyperplanes are defined in Section \ref{section_div_cube_complex}).

\Doublestate{Theorem}{div_bounds_theorem}{higher_degree_div_theorem}
{
Suppose $X$ is an essential, locally compact CAT(0) cube complex with cocompact automorphism group. Let $\mathcal{Y}$ and $\mathcal{Z}$ be non-intersecting hyperplanes in $X$.  
\begin{enumerate}[resume]
\item If $\mathcal{Y}$ and $\mathcal{Z}$ are $k$--separated, then $Div(X)$ is bounded below by a quadratic function. 
\item If $\mathcal{Y}$ and $\mathcal{Z}$ are $k$--chain separated, then $Div(X) \succeq \frac{1}{2}  R^2 \log_2(\log_2(R))$.
\item If $X$ contains a pair of degree $d$ $k$--separated hyperplanes, then $Div(X)$ is bounded below by a polynomial of degree $d+1$. \label{intro_div_thm3}
\item Suppose $X$ has $k$--alternating geodesics. If $\mathcal{Y}$ and $\mathcal{Z}$ are symbolically $k$--chain separated then $Div(X)$ is bounded below by a cubic function. \label{intro_div_thm4}
\end{enumerate}
}

The aforementioned characterization of quadratic divergence in right-angled Coxeter groups is an application of \ref{intro_div_thm4} above. Furthermore, as an application of \ref{intro_div_thm3} we give graph-theoretic criteria which imply polynomial lower bounds on divergence of right-angled Coxeter groups. 

\Restate{Theorem}{higher_degree_div_racg}{}
{
Suppose the graph $\Gamma$ contains a rank $n$ pair $(s, t)$, then $Div(W_{\Gamma})$ is bounded below by a polynomial of degree $n+1$. 
}

Here a rank $n$ pair, where $n$ is any nonnegative integer, is a pair of non-adjacent vertices of $\Gamma$ that satisfy a certain inductive graph-theoretic criterion. The above theorem together with the machinery of thickness (see \cite{BHSC} and \cite{Lev2}), provide exact bounds on the divergence of a wide range of right-angled Coxeter groups. The above theorem, in particular, applies to the examples given in \cite{DT}. However, this theorem is still not sufficient to characterize divergence in RACGs as demonstrated by Remark \ref{cubic_counterexample_rmk}.

Finally, we explore the divergence in the setting of Coxeter groups (not necessarily right-angled). Given an edge-labeled simplicial graph $\Gamma$ there is a corresponding Coxeter group $W_{\Gamma}$. An adaptation of our techniques allows us to prove results in this general case. For instance, we provide the following polynomial lower bound.

\Restate{Theorem}{coxeter_higher_div_thm}{}
{
Let $\Gamma$ be an even triangle-free Coxeter graph. Suppose $(u, v)$ is a rank $n$ pair, then the divergence of the Coxeter group $W_{\Gamma}$ is bounded below by a polynomial of degree $n+1$. 
}

By the above theorem and the results from \cite{BHSC}, for any positive integer degree, we can conclude there are infinite classes of Coxeter groups which are not right-angled and which have polynomial divergence of that degree. This shows the existence of higher degree polynomial divergence in the general class of Coxeter groups is abundant.

Theorem \ref{coxeter_higher_div_thm} is actually proven in a more general setting as we only need $\Gamma$ to be triangle-free and even for some neighborhood of the vertex $u$. For a precise statement see Theorem \ref{coxeter_higher_div_thm} of Section \ref{section_coxeter_groups}. For an edge-labeled graph $\Gamma$ representing a Coxeter group, we let $\hat{\Gamma}$ denote the graph obtained by collapsing odd labeled edges of $\Gamma$ to a point. We prove the following:

\Restate{Theorem}{coxeter_quadratic_thm}{}
{Let $\Gamma$ be a Coxeter graph. If the diameter of $\hat{\Gamma}$ is larger than 2, then $W_{\Gamma}$ has at least quadratic divergence. 
}

In particular, the above theorem shows that if $W_{\Gamma}$ is an even Coxeter group where $\Gamma$ has diameter larger than 2, then the divergence of $W_{\Gamma}$ is at least quadratic. 

The paper is organized as follows. Section \ref{section_background} provides general background material, including a background on divergence, CAT(0) cube complexes and Coxeter groups. Section \ref{section_disk_diagrams} provides necessary background on disk diagrams in CAT(0) cube complexes. In this section we set the notation and results regarding disk diagram structures that are used throughout the article. 

In Section \ref{section_separation_props}, we introduce several notions of separation for a pair of non-intersecting hyperplanes in a CAT(0) cube complex. The consequences of these separation properties on the divergence of a CAT(0) cube complex are explored in Section \ref{section_div_cube_complex} and \ref{section_higher_bounds}. The hyperplane divergence function is defined there as well. 

In Section \ref{section_racg_divergence} we apply the results obtained for CAT(0) cube complexes to the setting of right-angled Coxeter groups. Finally, in Section \ref{section_coxeter_groups}, we explore the divergence of Coxeter groups which are not necessarily right-angled. 

\vspace{.2in}

\noindent \textbf{Acknowledgements:} I would like to especially thank my advisor Jason Behrstock for countless discussions and guidance in writing this paper. I would also like to thank Mark Hagen for many useful suggestions and pointing me in the right directions in the world of CAT(0) cube complexes. Finally, I am also grateful to the anonymous referee for the many helpful suggestions and corrections.

\section{Background} \label{section_background}

Given a metric space $X$, we will always use $B_p(r)$ to denote the ball of radius $r$ about the point $p \in X$. 

\begin{definition}[quasi-isometry] \label{quasi_isometry_def}

Let $X$ and $Y$ be metric spaces. A \textit{($k$, $c$)--quasi-isometry} is a not necessarily continuous map $f: X \to Y$, such that for all $a, b \in X$ we have:
\[ \frac{1}{k}d_X(a,b) - c \le d_Y(f(a), f(b)) \le kd_X(a,b) + c \]

\end{definition}

Quasi-isometries provide a natural notion of equivalence in a coarse geometric setting. For a detailed background on quasi-isometries and geometric group theory in general see \cite{BrHa}.

\subsection{Divergence}

	Let $X$ be a metric space. Fix constants $0 < \delta \le 1$, $\lambda \ge 0$ and consider the linear function $\rho(r) = \delta r - \lambda$. Let $a, b, c \in X$ and set $k = d(c, \{ a, b \})$.  
	
	\[ div_{\lambda}(a, b, c, \delta) \] is the length of the shortest path in $X$ from $a$ to $b$ which avoids the ball $B_c(\rho(k))$. 
	
	\[ Div_{\lambda}^X(r, \delta) \] 
	
	is the supremum of $div_{\lambda}(a, b, c, \delta)$ over all $a, b, c$ with $d(a, b) \le r$. We will often suppress $\delta$ and $\lambda$ from the notation and say the \textit{divergence} of $X$ is the function $Div(X) = Div_{\lambda}^X(r, \delta)$. 

We set $f(r) \asymp g(r)$ if there exists a $C$ such that: 
	
	 \[ \frac{1}{C} g(\frac{r}{C}) - Cr - C < f(r) < Cg(Cr) + Cr + C \] 
	 
Up to this equivalence relation on functions and under mild assumptions on the metric space, divergence is a quasi-isometry invariant. See Lemma 3.4 \cite{DMS} for the relevant hypotheses. For instance, the Cayley graph of a finitely generated group  satisfies such hypotheses. 

\begin{remark}
	Let $\gamma \subset X$ be a bi-infinite geodesic with base point $b$. The divergence of $\gamma$ is a function in $r$ whose values are the supremum of $div_{\lambda}(\gamma(r), \gamma(-r), b, \delta)$. Under the above equivalence relation on functions, this definition is well-defined regardless of the choice of base point. It follows the divergence of $X$ is bounded below by the divergence of $\gamma$.
\end{remark}

\subsection{CAT(0) Cube Complexes}

A \textit{CAT(0) cube complex}, $X$, is a simply connected cell complex whose cells consist of Euclidean unit cubes, $[ -\frac{1}{2}, \frac{1}{2}]^d$, of varying dimension $d$. Additionally, the link of each vertex is a flag complex (i.e., any set of vertices which are pairwise connected by an edge, span a simplex). A CAT(0) cube complex with the induced metric is a CAT(0) space (see \cite{CS} and \cite{Wise} for more detailed references). We say $X$ is \textit{finite-dimensional} if there is an upper bound on the dimension of cubes in $X$. 

A \textit{midcube} $Y \subset [ -\frac{1}{2}, \frac{1}{2}]^d$ is the restriction of a coordinate to 0. A \textit{hyperplane} $\mathcal{H} \subset X$ is a connected subspace with the property that for each cube $C$ in $X$, $\mathcal{H} \cap C$ is a midcube or $\mathcal{H} \cap C = \emptyset$. It follows $X - \mathcal{H}$ consists of exactly two distinct components. A \textit{half-space} is the closure of such a component. We denote the two half-spaces associated to $\mathcal{H}$ by $\mathcal{H}^+$ and $\mathcal{H}^-$. The carrier of a hyperplane, $N(\mathcal{H})$, is the set of all cubes in $X$ which have non-trivial intersection with $\mathcal{H}$. 

A CAT(0) cube complex $X$ is \textit{essential} if all its half-spaces contain arbitrarily large balls of $X$. If $X$ is one-ended and essential, then every hyperplane is unbounded.

We will work exclusively with the combinatorial metric on the $1$--skeleton of $X$. A \textit{combinatorial geodesic} is a geodesic in the $1$--skeleton of $X$ under this metric and a \textit{combinatorial path} is a path in the $1$--skeleton of $X$. We often drop the word ``combinatorial'' from these definitions. The following core lemmas, whose proofs are found in \cite{CS}, are used throughout this paper. 

\begin{lemma} \label{qi_combinatorial_lemma}
Suppose $X$ is a finite-dimensional CAT(0) cube complex. $X$ is quasi-isometric to its 1--skeleton endowed with the combinatorial metric. 
\end{lemma}

\begin{lemma} \label{ramsey_lemma}
Suppose $X$ is a finite-dimensional CAT(0) cube complex. For each $k > 0$, there exists a number $N(k)$ such that any combinatorial geodesics of length $N(k)$ in $X$ must cross a set of pairwise non-intersecting hyperplanes $\{\mathcal{H}_1, ..., \mathcal{H}_k \}$. 

\end{lemma}

\subsection{Coxeter Groups}

We only give a brief background on Coxeter groups. For an extensive background we refer the reader to \cite{BB} and \cite{Dav}.

In this paper $\Gamma$ will always denote a simplicial graph (usually a Coxeter graph). The vertex set and edge set of $\Gamma$ are respectively denoted by $V(\Gamma)$ and $E(\Gamma)$. A \textit{clique} in $\Gamma$ is a subgraph whose vertices are all pairwise adjacent.  A $k$--clique is a clique with $k$ vertices. 

For $s \in \Gamma$, the \textit{link of $s$}, $Link(s) \subset V(\Gamma)$, is the set of vertices in $\Gamma$ connected to $s$ by an edge. The \textit{star of $s$} is the set $Star(s) = Link(s) \cup s$.

A Coxeter group is defined by the presentation: 
\[W = \langle s_1,s_2,...,s_n | (s_is_j)^{m(s_i,s_j)} = 1 \rangle \]
where $m(s_i,s_i) = 1$ and $m(s_i, s_j) = m(s_j, s_i) \in \{2,3,...,\infty\}$ when $i \ne j$. If $m(s_i, s_j) = \infty$ then no relation of the form $(s_is_j)^m=1$ is imposed. 

Given a presentation for a Coxeter group, there is a corresponding labeled \textit{Coxeter graph} $\Gamma$. The vertices of $\Gamma$ are elements of $S$. There is an edge between $s_i$ and $s_j$ if and only if $m(s_i,s_j) \ne \infty$. This edge is labeled by $m(s_i, s_j)$ if $m(s_i,s_j) \ge 3$. If $m(s_i, s_j) = 2$ no label is placed on the corresponding edge. Conversely, given such an edge-labeled graph $\Gamma$, we have the Coxeter group $W_{\Gamma}$. 

In the literature, there are many different conventions for associating a graph to a Coxeter group presentation. The given convention was chosen to make the theorems in this paper easier to state. 

A \textit{right-angled Coxeter group (RACG)} is a Coxeter group with generating set $S$ where $m(s, t) \in \{ \infty, 2 \}$ for $s, t$ distinct elements in S. An \textit{even Coxeter group} is a Coxeter group given by a Coxeter graph where each edge either has an even label or no label. 

We will often want to consider subgroups of a Coxeter group $W_{\Gamma}$ corresponding to subgraphs of $\Gamma$. The full subgraph of $T \subset V(\Gamma)$ is the graph with vertex set $T$ with a labeled edge $(t_1, t_2)$ if and only if $(t_1, t_2)$ is an edge of $\Gamma$ with the same label.

\begin{definition} [Induced Subgroup] \label{induced_subgroup_def}
Let $(W,S)$ be a Coxeter group. For $T \subset S$, let $W_T$ be the subgroup of $W$ generated by the full subgraph of $T$. This subgroup is isomorphic to the Coxeter group corresponding to the subgraph of $\Gamma$ induced by $T$ (see \cite{BB} for instance).
\end{definition}

The \textit{Davis complex}, $\Sigma_{\Gamma}$, is a natural CAT(0) cell complex which the Coxeter group $W_{\Gamma}$ acts geometrically. In this paper, we will only make use of the Davis complex for right-angled Coxeter groups. 

Suppose $W_{\Gamma}$ is a RACG.  For every $k$--clique, $T \subset \Gamma$, the induced subgroup $W_T$ is isomorphic to the direct product of $k$ copies of $\mathbb{Z}_2$. Hence, the Cayley graph of $W_T$ is isomorphic to a $k$--cube. The Davis complex $\Sigma_{\Gamma}$ has 1-skeleton the Cayley graph of $W_{\Gamma}$ where edges are given unit length. Additionally, for each $k$--clique, $T \subset \Gamma$, and coset, $gW_T$, we glue a unit $k$--cube to $gW_T \subset \Sigma_{\Gamma}$. It is then evident that the Davis Complex for a RACG is naturally a CAT(0) cube complex. 

\section{Disk Diagrams in CAT(0) Cube Complexes} \label{section_disk_diagrams}

A \textit{disk diagram} $D$ is a contractible finite $2$--dimensional cube complex with a fixed planar embedding $P: D \to \mathbb{R}^2$. The \textit{area} of $D$ is the number of $2$--cells it contains. By compactifying $\mathbb{R}^2$, $S^2 = \mathbb{R}^2 \cup \infty$, we can extend $P: D \to S^2$, giving a cellulation of $S^2$. The \textit{boundary path of} $D$, $\partial D$, is the attaching map of the cell in this cellulation containing $\infty$. Note that this is not necessarily the topological boundary.

Let $X$ be a CAT(0) cube complex. We say $D$ is a \textit{disk diagram in} $X$, if $D$ is a disk diagram and there is a fixed continuous combinatorial map of cube complexes $F: D \to X$. By a lemma of Van Kampen, for every null-homotopic closed combinatorial path $p: S^{1} \to X$, there exists a disk diagram $D$ in $X$ such that $\partial D = p$. 

Suppose $D$ is a disk diagram in a CAT(0) cube complex $X$ and $t$ is a $1$--cell of $D$. A \textit{dual curve $H$ dual to $t$}, is a concatenation of midcubes in $D$ which contains a midcube in $D$ which intersects $t$. The image of $H$ under the map $F: D \to X$ lies in some hyperplane $\mathcal{H} \subset X$. We also have that every edge in $D$ is dual to exactly one maximal dual curve. 

\textbf{In our notation, we denote a dual curve by a capital letter and its corresponding hyperplane by the corresponding script letter.}

An \textit{end} of a dual curve $H$ in $D$ is a point of intersection of $H$ with $\partial D$. Maximal dual curves either have no ends or two ends. The \textit{carrier} $N(H)$ of $H$ is the set of 2-cubes in $D$ containing $H$.  

Suppose we have oriented combinatorial paths $p_1$, $p_2$, ..., $p_n$ in a disk diagram $D$ and that $p_i \cap p_{i+1} \ne \emptyset$ for $1 \le i < n$. We then define a new oriented path $p = p_1 * p_2 * ... * p_n$ by beginning at the first point of $p_1$, followed by $p_1$ until its first intersection with $p_2$; followed by $p_i$ until its first intersection with $p_{i+1}$. In this definition, we further assume the orientations are chosen such that this construction is possible (i.e., we can always follow $p_i$ along its orientation until its intersection with $p_{i+1}$). Furthermore, different choices of path orientations could produce a different paths. We note that this construction is only used in Lemma \ref{chain_separated_div_lemma} and Lemma \ref{symbolically_separated_div_lemma}, and the relevant path orientations there are given.

Many of the ideas in this section originated in \cite{Hag} and \cite{Wise}. For our purposes, we require some modified definitions and lemmas to those in the mentioned works. For completeness we include proofs of these modified claims, even though many of the arguments are very similar. 

To every closed loop formed by a concatenation of combinatorial paths and hyperplanes, we wish to associate a disk diagram with boundary path this loop. This notion is formally defined below. If such a diagram is chosen appropriately, the dual curves associated to it behave nicely. We call such nicely behaved diagrams \textit{combed} and define them later in this section.

\begin{definition} [Hyperplane-Path Sequence] \label{hyperplane_path_sequence_def}
In the following definition, we work modulo $n+1$ (i.e, $n+1 = 0$). Let $\bar{A} = \{A_0, A_1, ..., A_n \}$ be a sequence such that for each $i$, $A_i$ is either a hyperplane or an oriented combinatorial path in a CAT(0) cube complex $X$. Furthermore, for $0 \le i \le n$, if $A_i$ and $A_{i+1}$ are both hyperplanes then they intersect. If $A_i$ and $A_{i+1}$ are both combinatorial paths, then the endpoint of $A_i$ is the initial point of $A_{i+1}$. If $A_i$ is a hyperplane and $A_{i+1}$ is a path, then the beginning point of $A_{i+1}$ lies on $N(A_i)$. Similarly, if $A_i$ is a path and $A_{i+1}$ is a hyperplane, then the endpoint of $A_{i}$ lies on $N(A_{i+1})$. We call $\bar{A}$ a \textit{hyperplane-path sequence}. 

Given a hyperplane-path sequence $\bar{A} = \{A_0, ..., A_n \}$, let $\bar{P} = \{P_0, ..., P_n\}$ be a sequence of combinatorial paths where $P_i = A_i$ if $A_i$ is a combinatorial path and $P_i$ is a combinatorial geodesic in $N(A_i)$ if $A_i$ is a hyperplane. Furthermore, assume $P = P_0 * P_1 * ... * P_n$ defines a loop. A disk diagram  $D$ is \textit{supported by $\bar{A}$} if $\partial D = P$ for some choice of $\bar{P}$. Often this choice is given and we say \textit{$D$ is supported by $\bar{A}$ with boundary path $\bar{P}$}. For an example of a disk diagram supported by a hyperpane path sequence see Figure \ref{fig_k_separated_lemma}.

\end{definition}

\begin{remark}
Diagrams supported by a hyperplane-path sequence are a special case of diagrams with fixed carriers defined in \cite{Hag}. The difference is that consecutive hyperplanes in a hyperplane-path sequence must intersect, where in \cite{Hag} they either intersect or osculate. Most of this section can be modified to allow for osculating hyperplanes; however, there was no need for such sequences in this paper.  
\end{remark}

\begin{definition}[Nongons, Bigons, Monogons and Oscugons] \label{gons_def}
A \textit{nongon} is a dual curve of length greater than one which begins and ends on the same dual $1$--cell. A \textit{bigon} is a pair of dual curves which intersect at their first and last containing squares. A \textit{monogon} is a dual curve which intersects itself in its first and last square. An \textit{oscugon} is a dual curve which starts at the dual $1$--cell $e_1$, ends at the dual $1$--cell $e_2$, such that $e_1 \neq e_2$, $e_1 \cap e_2 \neq \emptyset$ and $e_1, e_2$ are not contained in a common square. A disk diagram \textit{without pathologies} is one that does not contain a nongon, bigon, monogon or oscugon. 
\end{definition}

The following is proved in \cite[Corollary~2.4]{Wise}: 

\begin{lemma}[\cite{Wise}] \label{no_monogons_lemma}
Suppose $D$ is a disk diagram in a CAT(0) cube complex $X$, then $D$ does not contain monogons. 	
\end{lemma}

We now wish to discuss the idea of \textit{boundary combinatorics} in a disk diagram $D$. Suppose $D$ and $D'$ are two disk diagrams in a CAT(0) cube complex $X$. Let $p \subset \partial D$ and $p' \subset \partial D'$ be subcomplexes. We say $p$ and $p'$ are \textit{equal boundary complexes} if the canonical maps $p \to X$ and $p' \to X$ are the same combinatorial maps. 

Suppose $p \subset \partial D$ and $p' \subset \partial D'$ are equal boundary complexes, and let $i: p \to p'$ be the canonical isomorphism between them. We say $p$ and $p'$ have \textit{equal boundary combinatorics} if given any pair of edges $e_1, e_2 \in p$ dual to a common dual curve in $D$, it follows $i(e_1)$ and $i(e_2)$ are dual to a common dual curve in $D'$. 

In particular, $D$ and $D'$ have \textit{equal boundary} if $\partial D$ and $\partial D'$ are equal combinatorial complexes. Additionally, $D$ and $D'$ have \textit{equal boundary combinatorics} if $\partial D$ and $\partial D'$ have equal boundary combinatorics. 

The following is also proved in \cite[Lemma~2.3]{Wise}: 

\begin{lemma} [\cite{Wise}] \label{wise_lemma}
Let $D$ be a disk diagram in a CAT(0) cube complex $X$. There exists a disk diagram $D'$ in $X$  with no pathologies and equal boundary combinatorics to $D$.
\end{lemma}

\begin{definition}[Combed Diagram] \label{combed_diagram_def}
Let $D$ be a disk diagram supported by hyperplane-path sequence $\bar{A} = \{ A_0, ..., A_n \}$ with boundary path $\bar{P} = \{P_0, ..., P_n\}$. We say $D$ is \textit{combed} if the following properties hold:

\begin{enumerate} 
\item $D$ has no pathologies.
\item If $A_i$ is a hyperplane, no two dual curves dual to $P_i$ intersect. In particular, no dual curve has both ends on $P_i$.
\item If $A_i$ and $A_{i+1}$ are both hyperplanes, no dual curve dual to $P_i$ intersects $P_{i+1}$.
\end{enumerate}
\end{definition}

The arguments in the next two lemmas are essentially the same as those in the proof of \cite[Lemma~2.11]{Hag}. However, for our purposes, we often require a statement regarding the boundary combinatorics of a given disk diagram. To be self-contained and to guarantee the arguments are valid in our context, we provide them here.

\begin{lemma} \label{combed_prop_lemma}
Suppose $D$ is a disk diagram supported by hyperplane-path sequence $\bar{A} = \{ A_0, ..., A_n \}$ with boundary path $\bar{P} = \{P_0, ..., P_n\}$, satisfying properties $1$ and $2$ of combed diagrams (Definition \ref{combed_diagram_def}). There exists a combed diagram $D'$ with boundary path $\bar{P}' = \{ P_0', ..., P_n'\}$, where $P_i'$ is a connected subsegment of $P_i$. Furthermore, $\bar{P}'$ has equal boundary combinatorics as its image in $\bar{P}$. 
\end{lemma}
\begin{proof}
Suppose some dual curve $C$ intersects both $P_i$ and $P_{i+1}$.  Let $v$ be the vertex where $P_i$ meets $P_{i+1}$. Let $e_1$ be the edge in $P_i$ which intersects $v$ and $e_2$ the edge in $P_{i+1}$ which intersects $v$. 

By property 2 of combed diagrams, every dual curve to $P_i$ between $v$ and $C$ intersects $P_{i+1}$ as well. It follows that the dual curve, $K$, to $e_1$ intersects $e_2$. Let $Q$ be the combinatorial path in $N(K) - K$ that forms a loop based at $v$. Note that $Q$ cannot contain any edge of $\partial D$.  If $e_1 \neq e_2$, then any dual curve to $Q$ must intersect $Q$ twice, forming a bigon. However, as $D$ does not contain bigons, it follows $e_1 = e_2$. Hence, we obtain a new diagram from $D$ by simply deleting the edge $e_1$. By repeating this process we obtain the desired diagram $D'$. 

\end{proof}

\begin{lemma} \label{remove_intersecting_curves_lemma}
Let $D$ be a disk diagram in a CAT(0) cube complex $X$ supported by hyperplane-path sequence $\bar{A} = \{ A_0, ..., A_n \}$ with boundary path $\bar{P} = \{P_0, ..., P_n \}$. Suppose for some $i$, $\mathcal{A} = A_i$ is a hyperplane. There exists a disk diagram $D'$ also supported by $\bar{A} = \{ A_0,  ..., A_n \}$ with boundary path $\bar{P}' = \{ P_0, ..., P_{i-1}, P_i', P_{i+1}, ..., P_n \}$ such that $\partial D - P_i$ has the same boundary combinatorics as $\partial D' - P_i'$. Additionally, no two dual curves dual to $P_i'$ in $D'$ intersect. 
\end{lemma}

\begin{proof}
By Lemma \ref{wise_lemma} we may assume $D$ has no pathologies. Set $P = P_i$ and suppose two dual curves to $P$, $C_1$ and $C_2$, intersect. 

A dual curve cannot have two ends on $P$. For otherwise $P$ would cross the same hyperplane twice, contradicting $P$ being geodesic. In particular, $C_1 \neq C_2$. Let $e_1$ be the edge on $P$ dual to $C_1$ and $e_2$ the edge on $P$ dual to $C_2$. If $d_{P}(e_1, e_2) = d >0$, it follows there is another dual curve to $P$, $C_3$, between $e_1$ and $e_2$. $C_3$ must then either intersect $C_1$ or intersect $C_2$ ($C_3$ cannot have both ends on $P$). Proceeding this way, we can then assume that $e_1$ and $e_2$ are distinct adjacent edges. 

Let $S$ be a square where $C_1$ and $C_2$ intersect. There are two cases. First suppose that $S$ does not contain $e_1$ and $e_2$ as edges. As $C_1$ intersects $C_2$, it follows, $e_1$ and $e_2$ must both lie in another square of $X$, say $S'$. For if this were not the case, the hyperplanes associated to the dual curves $C_1$ and $C_2$ would both cross and osculate (i.e., have dual adjacent edges that are not in a common square). However, this is not possible in a CAT(0) cube complex (see \cite[Section 6b]{Wise}).

As the link of vertices in $X$ are flag complexes, it follows $S' \subset N(\mathcal{A})$. We can then form a new disk diagram, $D'$, by attaching $S'$ to $D$ along the edges $e_1$ and $e_2$. This modifies the path $P$ into a new path $P'$, that is still geodesic and is still in $N(\mathcal{A}) $. In $D'$, the dual curves $C_1$ and $C_2$ now form a bigon. By \cite[Lemma 2.3]{Wise}, there is a another disk diagram $D''$, with the same boundary combinatorics as $D'$ and no pathologies, such that $\text{Area}(D'') \le \text{Area}(D') - 2 = \text{Area}(D) -1$. We have thus produced a diagram, $D''$, with the desired boundary combinatorics that is of area strictly smaller than $D$.

For the second case, suppose $e_1$ and $e_2$ are edges of $S$. Label the other edges of $S$ as $e_3$ and $e_4$. Let $\mathcal{H}_1$ and $\mathcal{H}_2$ respectively be hyperplanes which pass through $e_1$ and $e_2$. If $\mathcal{A} = \mathcal{H}_1$, then it follows that $e_1, e_2, e_3 \subset N(\mathcal{A})$. Hence, $S \subset N(\mathcal{A})$. Alternatively, suppose $\mathcal{H}_1$, $\mathcal{H}_2$, and $\mathcal{A}$ are distinct. Since the link of vertices in $X$ are flag complexes, it follows $S$ is in a $3$-cube which is contained in $N(\mathcal{A})$. Either way, $S \subset N(\mathcal{A})$. 

Let $P'$ be the path in $D$ which is the same as $P$ with $e_1, e_2$ replaced with $e_3, e_4$. $P'$ is still geodesic. Furthermore, $P' \subset N(\mathcal{A})$. Let $D' \subset D$ be the disk diagram obtained as a subdiagram of $D$ by replacing $P$ with $P'$. It follows $D'$ is a diagram of smaller area than $D$. Furthermore, the boundary combinatorics of $\partial D - P$ are not affected. 

In both cases we are able to produce a smaller area disk diagram with the desired boundary combinatorics. Therefore, by iterating this process we are guaranteed to eventually have a diagram with the conclusion of the lemma.
\end{proof}

The following lemma guarantees the existence of a combed diagram. 

\begin{lemma} \label{minimal_area_lemma}
Let $D$ be a disk diagram supported by hyperplane-path sequence $\bar{A}$. There exists a combed disk digram $D'$ also supported by $\bar{A}$. 
\end{lemma} 

\begin{proof}
The lemma follows by applying Lemma \ref{wise_lemma} to $D$ to get a disk diagram with no pathologies. We then repeatedly apply Lemma \ref{remove_intersecting_curves_lemma} to the resulting diagram to obtain a diagram satisfying property 1 and 2 of combed diagrams. Finally, we apply  Lemma \ref{combed_prop_lemma} to obtain a combed diagram. 
\end{proof}

Given a maximal dual curve $C$ in a disk diagram $D$, we want to construct a new combed diagram with the hyperplane $\mathcal{C}$ in its support. Furthermore, we want the appropriate boundary combinatorics of $D$ preserved in this new diagram. The following technical lemma guarantees the existence of such a diagram.

\begin{lemma} \label{new_combed_diagram_lemma}
Let $D$ be a combed disk diagram in a CAT(0) cube complex $X$ supported by hyperplane-path sequence $\bar{A} = \{ A_0, ...,A_n \}$ and with boundary path $\bar{P} = \{P_0, ..., P_n\}$. 

Suppose $C$ is a maximal dual curve from $P_i$ to $P_j$ with $i < j$ and let $\mathcal{C} \subset X$ be the hyperplane associated to $C$. Let $P$ be a combinatorial path in $N(C)$ from $P_i$ to $P_j$. Let $P_i'$ be the subsegment of $P_i$ between $P_{i-1}$ and $P$, and let $P_j'$ be the subsegment of $P_j$ between $P$ and $P_{j+1}$. Let $A_i'$ and $A_j'$ be the corresponding supports of $P_i'$ and $P_j'$ respectively ($A_i' = \mathcal{A}_i$ if $A_i$ is a hyperplane and $A_i = P_i'$ otherwise). 

Let $P'$ be any combinatorial geodesic in $N(\mathcal{C})$ connecting the endpoints of $P$. There exists a combed disk diagram $D'$ supported by 
\[ \bar{A}' = \{A_0, ..., A_{i-1}, A_i', \mathcal{C}, A_j', A_{j+1}, ..., A_n \} \] 

with boundary path

\[ \bar{P}' = \{P_0, ..., P_{i-1}, P_i', P', P_j',P_{j+1} ..., P_n\}\]

 such that $\partial D' - P'$ has the same boundary combinatorics as the corresponding subset of $\partial D$.  
\end{lemma}

\begin{proof}
Let $D_1$ be a combed diagram with boundary path $\{ P', P \}$. Let $D_2$ be the subdiagram of $D$ with boundary path $\{ P_0, ...,P_i', P, P_{j}',...,P_n \}$. Note that $D_2$ is still combed.  We may form a new disk diagram $D_3$ by gluing $D_1$ to $D_2$ along $P$. $D_3$ is supported by $\{ A_0, ...,A_i', \mathcal{C}, A_{j}',...,A_n \}$ with boundary path $\{P_0, ..., P_i', P', P_{j}',... P_n \}$. 

By Lemma \ref{wise_lemma}, we may assume $D_3$ has no pathologies and $\partial D_3 - P'$ has the same combinatorics as the corresponding subset of $\partial D$.  

All that is left to prove is that properties 2 and 3 of combed diagrams (Definition \ref{combed_diagram_def}) hold. Property 3 clearly still holds for dual curves which do not intersect $P'$. Assume $A_i$ is a hyperplane. Since $D$ is combed, no dual curve to $P_i'$ in $D_2$ intersects $P$. Since the combinatorics of $D_2$ are preserved in $D_3$, this is still the case in $D_3$. In particular, no dual curve to $P_i'$ intersects $P'$ in $D_3$. Therefore, property 3 holds in $D_3$. 

Assume property 2 is false in $D_3$. We then have two intersecting dual curves, $C_1$ and $C_2$, that are dual to the same boundary path in $D_3$ with hyperplane support. By the preservation of boundary combinatorics and the fact that $D_3$ is combed, it follows each of these curves must have an endpoint on $P'$. We can then modify $D_3$ using Lemma \ref{remove_intersecting_curves_lemma} to obtain a new diagram $D'$ satisfying the claim. 
\end{proof}

\section{Hyperplane Separation Properties} \label{section_separation_props}

\textbf{For the remainder of the paper, $X$ will denote a CAT(0) cube complex. Furthermore, $\mathcal{Y}$ and $\mathcal{Z}$ will always denote a pair of non-intersecting unbounded hyperplanes in $X$.} 

We will discuss different definitions for separation properties of a given pair of non-intersecting hyperplanes. In the next section we will explore the relationship between these separation properties and the divergence of $X$.

Definition \ref{strongly_separated_def} given below is from \cite{BC}.

\begin{definition} [\cite{BC}] \label{strongly_separated_def}
$\mathcal{Y}$ and $\mathcal{Z}$ are \textit{strongly separated} if no hyperplane intersects them both. 
\end{definition}

A \textit{minimal geodesic} $g$ between hyperplanes $\mathcal{Y}$ and $\mathcal{Z}$ is a combinatorial geodesic with endpoints on $N(\mathcal{Y})$ and $N(\mathcal{Z})$ such that $|g|$ is minimal over all such geodesics. The following lemma shows minimal geodesics between strongly separated hyperplanes have the same endpoints.

\begin{lemma} \label{strongly_separated_lemma}
	Suppose $\mathcal{Y}$ and $\mathcal{Z}$ are strongly separated hyperplanes. Let $g_1$ and $g_2$ be minimal geodesics between $\mathcal{Y}$ and $\mathcal{Z}$. It follows that $N(\mathcal{Y}) \cap g_1 = N(\mathcal{Y}) \cap g_2$. Consequently, $g_1$ and $g_2$ have the same endpoints.
\end{lemma}

\begin{proof}
	Let $N(\mathcal{Y}) \cap g_1 = v_1$ and $N(\mathcal{Y}) \cap g_2 = v_2$. Suppose, for a contradiction, that $v_1 \neq v_2$. Let $h$ be a combinatorial geodesic from $v_1$ to $v_2$ in $N(\mathcal{Y})$ and let $\mathcal{H}$ be a hyperplane intersecting $h$. $\mathcal{H}$ cannot intersect $\mathcal{Z}$ since $\mathcal{Y}$ and $\mathcal{Z}$ are strongly separated. It follows $\mathcal{H}$ must either intersect $g_1$ or $g_2$. This is a contradiction (see \cite[Remark 3.12]{Wise}).
\end{proof}

The following definition gives a slight generalization of strongly separated hyperplanes.

\begin{definition} \label{k_separated_def}
$\mathcal{Y}$ and $\mathcal{Z}$ are $k$--separated if at most $k$ hyperplanes intersect both $\mathcal{Y}$ and $\mathcal{Z}$. In particular, a pair of strongly separated hyperplanes are $0$--separated.
\end{definition}

The following two lemmas describe how minimal geodesics and hyperplanes intersecting a pair of $k$--separated hyperplanes behave nicely.

\begin{lemma} \label{k_separated_lemma}
Suppose $\mathcal{Y}$ and $\mathcal{Z}$ are $k$--separated, then $d(\mathcal{H}_1 \cap \mathcal{Y}, \mathcal{H}_2 \cap \mathcal{Y}) \le k$ for every pair of hyperplanes $\mathcal{H}_1$, $\mathcal{H}_2$ which intersect both $\mathcal{Y}$ and $\mathcal{Z}$. Furthermore, either $\mathcal{Y}$ and $\mathcal{Z}$ are strongly separated, or every minimal geodesic $g$ connecting $\mathcal{Y}$ to $\mathcal{Z}$ lies in the carrier of a hyperplane which intersects both $\mathcal{Y}$ and $\mathcal{Z}$.
\end{lemma}

\begin{proof}
Suppose hyperplanes $\mathcal{H}_1$ and $\mathcal{H}_2$ each intersect both $\mathcal{Y}$ and $\mathcal{Z}$. For a contradiction, suppose that $d(\mathcal{H}_1 \cap \mathcal{Y}, \mathcal{H}_2 \cap \mathcal{Y}) > k$. Let $D$ be a combed disk diagram supported by $\{ \mathcal{Y}, \mathcal{H}_1, \mathcal{Z}, \mathcal{H}_2 \}$. Every dual curve to $\mathcal{Y}$ must intersect $\mathcal{Z}$. However, there are more than $k$ such dual curves and hence more than $k$ hyperplanes intersecting both  $\mathcal{Y}$ and $\mathcal{Z}$. This contradicts $\mathcal{Y}$ and $\mathcal{Z}$ being $k$-separated. 

To prove the lemma's second claim, suppose $g$ is a minimal geodesic from $\mathcal{Y}$ to $\mathcal{Z}$, and assume that $\mathcal{Y}$ and $\mathcal{Z}$ are not strongly separated. Suppose $\mathcal{H}$ is a hyperplane intersecting both $\mathcal{Y}$ and $\mathcal{Z}$ and let $D$ be a combed diagram supported by the hyperplane-path sequence $\{\mathcal{Y}, g, \mathcal{Z}, \mathcal{H} \}$ with boundary path $\{Y, g, Z, H\}$. Every dual curve to $H$ must intersect $g$. Since $g$ is geodesic, $|g| = |H|$ and it follows no dual curve to $\mathcal{Y}$ can intersect $g$. So every dual curve to $Y$ intersects $Z$. It follows $D$ is an Euclidean rectangle. Hence, there is a dual curve $C$ from $Y$ to $Z$ which has $g$ as part of its boundary path. So, $g$ is contained in the carrier of the hyperplane $\mathcal{C}$ which intersects both $\mathcal{Y}$ and $\mathcal{Z}$. 
\end{proof}

\begin{definition}
If infinitely many hyperplanes intersect both $\mathcal{Y}$ and $\mathcal{Z}$ then we say $\mathcal{Y}$ and $\mathcal{Z}$ are \textit{$\infty$-connected}. 
\end{definition}

A pair of non-intersecting hyperplanes are $k$--chain connected (formally defined below) if there is an appropriate sequence of sets of hyperplanes connecting them. Hyperplanes that are not $k$--chain connected, $k$--chain separated hyperplanes, provide a generalization of the notion of $k$--separated hyperplanes. 

\begin{figure}[h]
\centering
\begin{overpic}[scale=.35]{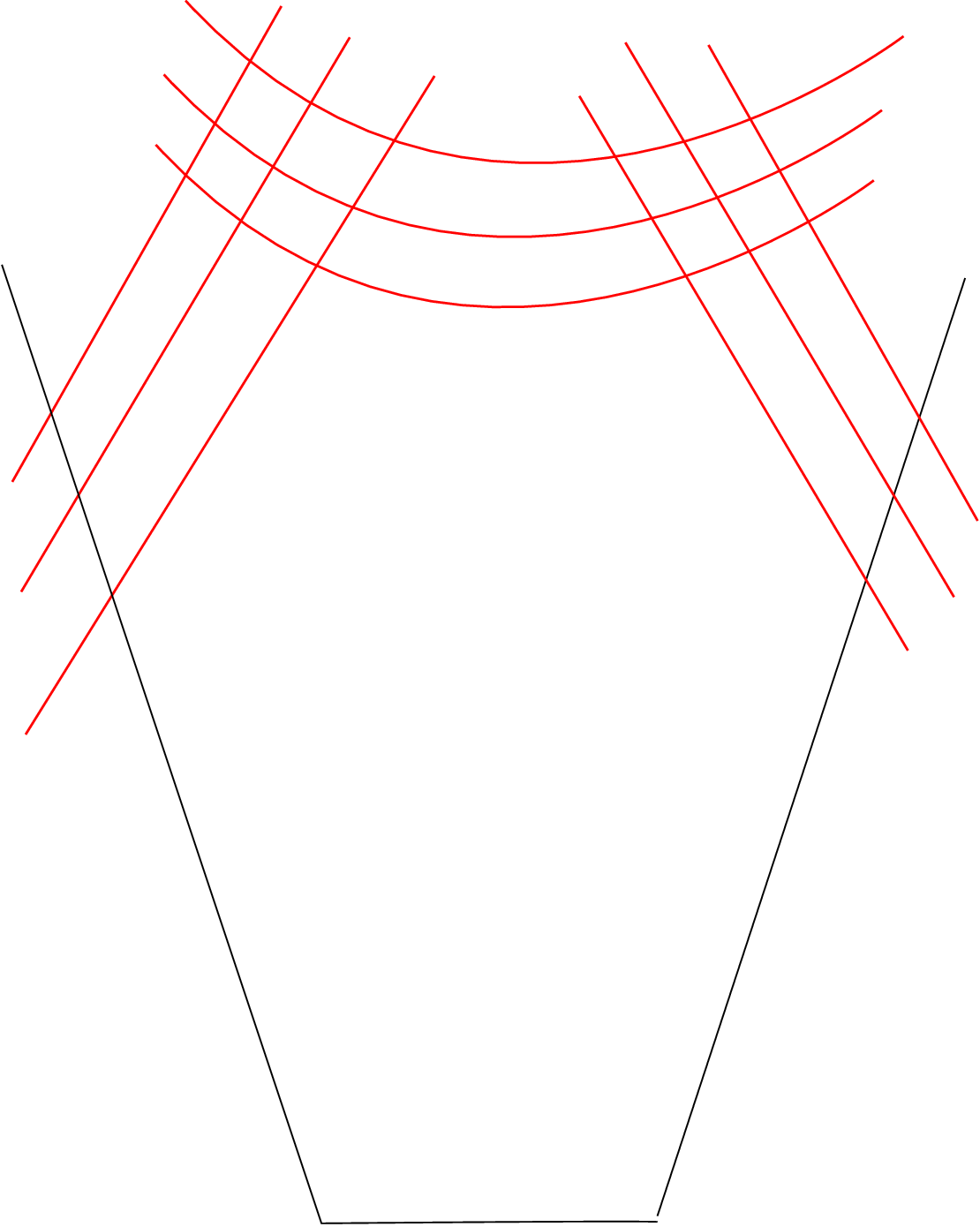}
\put(12,20){$\mathcal{Y}$}
\put(65, 20){$\mathcal{Z}$}
\put(40, 1){$g$}

\put(2, 75){$\mathcal{H}_1^1$}
\put(7, 68){$\mathcal{H}_2^1$}
\put(12, 60){$\mathcal{H}_3^1$}

\put(38, 88){$\mathcal{H}_1^2$}
\put(38, 81){$\mathcal{H}_2^2$}
\put(38, 71){$\mathcal{H}_3^2$}

\put(68, 73){$\mathcal{H}_1^3$}
\put(63, 68){$\mathcal{H}_2^3$}
\put(58, 63){$\mathcal{H}_3^3$}

\end{overpic}

\caption{The hyperplanes $\mathcal{Y}$ and $\mathcal{Z}$ above are $3$--chain connected.}
\end{figure}
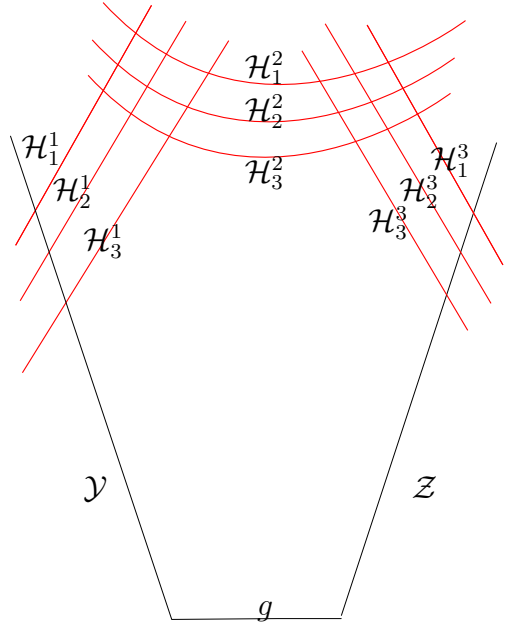

\begin{definition}[$k$--chain connected] \label{thickly_connected_hyperplanes_def}
$\mathcal{Y}$ and $\mathcal{Z}$ are \textit{$k$--chain connected} if there exists a sequence of length $k$ sequences of hyperplanes:  

\[ S_1 = \{ \mathcal{H}_{1}^{1}, \mathcal{H}_{2}^{1}, ..., \mathcal{H}_{k}^{1} \} \]
\[ S_2 = \{ \mathcal{H}_{1}^{2}, \mathcal{H}_{2}^{2}, ..., \mathcal{H}_{k}^{2} \} \] 
\[ ... \]
\[ S_m = \{ \mathcal{H}_{1}^{m}, \mathcal{H}_{2}^{m}, ..., \mathcal{H}_{k}^{m} \} \]

satisfying the following properties:
 
\begin{description}
\item[I.] For each $i$, hyperplanes in $S_i$ pairwise do not intersect. 
\item[II.] For each $i < m$, each hyperplane in $S_i$ intersects each hyperplane in $S_{i+1}$. 
\item[III.] Every hyperplane in $S_1$ intersects $\mathcal{Y}$ and every hyperplane in $S_m$ intersects $\mathcal{Z}$. 
\end{description}

\end{definition}

\begin{definition} \label{k_chain_separated_def}
$\mathcal{Y}$ and $\mathcal{Z}$ are \textit{$k$--chain separated} if they are not $k$--chain connected. 
\end{definition}

\begin{definition}
The hyperplanes $\mathcal{H}$ and $\mathcal{H}'$ are of the same type, if they are in the same orbit of $Aut(X)$. The hyperplanes $\mathcal{H}$ and $\mathcal{H}'$ are of non-intersecting type if $g\mathcal{H} \cap \mathcal{H}' = \emptyset$ for all $g \in Aut(X)$. 
\end{definition}

\begin{definition}
Let $Q = \{ \mathcal{H}_1, ..., \mathcal{H}_m \}$ be a sequence of hyperplanes. Define 
\[Type(Q) = (T_1, ..., T_m ) \]
where $T_i$ is the hyperplane type (orbit class) of $\mathcal{H}_i$. Note that the tuple $Type(Q)$ is ordered.
\end{definition}

The next set of definitions provide a further strengthening of the notion of $k$--chain separated hyperplanes which allows us to prove stronger divergence bounds in the next section. The following definitions were created with the key example of right-angled Coxeter groups in mind.

\begin{figure}[h]
\centering
\begin{overpic}[scale=.3]{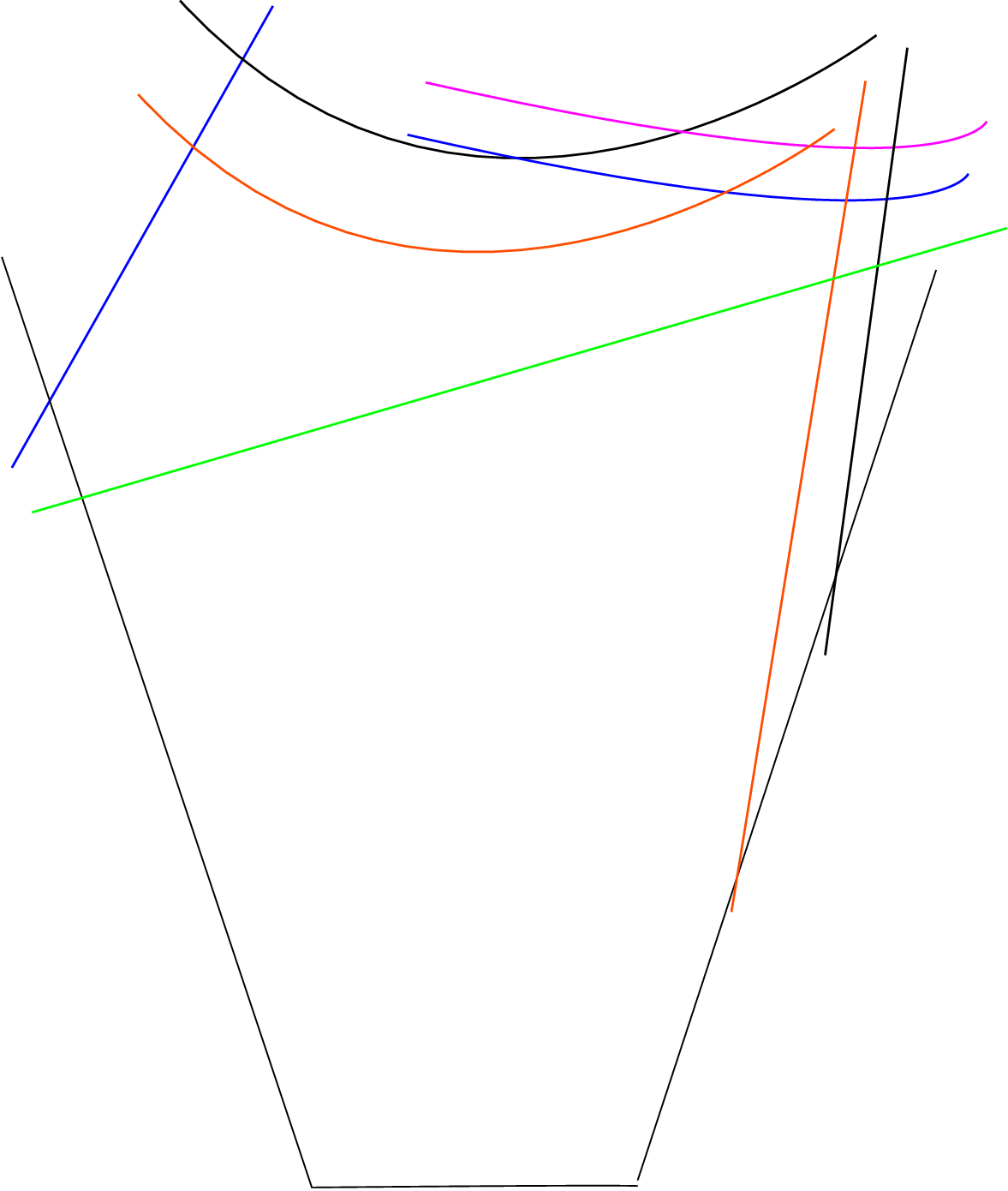}
\put(12,20){$\mathcal{Y}$}
\put(65, 20){$\mathcal{Z}$}
\put(40, 1){$g$}

\put(2, 75){$\mathcal{H}_1^1$}
\put(12, 60){$\mathcal{H}_2^1$}

\put(20, 90){$\mathcal{H}_1^2$}
\put(20, 81){$\mathcal{H}_2^2$}

\put(44, 92){$\mathcal{H}_1^3$}
\put(37, 88){$\mathcal{H}_2^3$}

\put(70, 69){$\mathcal{H}_1^4$}
\put(62, 62){$\mathcal{H}_2^4$}

\end{overpic}

\caption{Hyperplanes $\mathcal{Y}$ and $\mathcal{Z}$ are symbolically $2$--chain connected. The hyperplane colors signify their type.}
\end{figure}
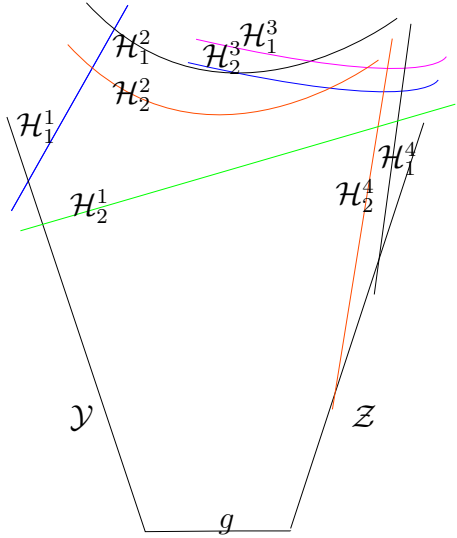

\begin{definition}[Symbolically $k$--chain Connected] \label{symbolically_chain_connected_def}
$\mathcal{Y}$ and $\mathcal{Z}$ are \textit{symbolically $k$--chain connected} if there exists a sequence of length $k$ sequences of hyperplanes:  

\[ S_1 = \{ \mathcal{H}_{1}^{1}, \mathcal{H}_{2}^{1}, ..., \mathcal{H}_{k}^{1} \} \]
\[ S_2 = \{ \mathcal{H}_{1}^{2}, \mathcal{H}_{2}^{2}, ..., \mathcal{H}_{k}^{2} \} \] 
\[ ... \]
\[ S_m = \{ \mathcal{H}_{1}^{m}, \mathcal{H}_{2}^{m}, ..., \mathcal{H}_{k}^{m} \} \]

satisfying the following five properties: 

\begin{itemize}
\item[I.] For each $i \le m$ and $j < k$, $H_j^i$ and $H_{j+1}^i$ are of non-intersecting type. 
\item[II.] For each $1 < i \le m$, every hyperplane in $S_{i}$ intersects $\mathcal{H}_{1}^{i-1}$. 
\item[III.] For each $i < m$ and $j \le k$, there exists an integer $c(i,j)$ such that $i < c(i,j) \le m$ and  $\mathcal{H}_{j}^{i}$ intersects every hyperplane in $S_{c(i,j)}$.

Additionally, for all $j, j' \le k$ and for all $i < m$, $Type(S_{c(i,j)}) = Type(S_{c(i,j')})$.
\item[IV.] Every hyperplane in $S_1$ intersects $\mathcal{Y}$ and every hyperplane in $S_m$ intersects $\mathcal{Z}$. 
\item[V.] Let $g$ be a minimal geodesic from $\mathcal{Y}$ to $\mathcal{Z}$. For all $i \le m$ and $j \le k$, $g$ and $\mathcal{H}_{j-1}^i$ lie in different half-spaces of $\mathcal{H}_{j}^i$. 
\end{itemize}

\end{definition}

\begin{remark}
By Lemma \ref{k_separated_lemma}, property V in the definition above necessarily implies that $m >1$. Furthermore, it follows that if property V is true for a minimal geodesic from $\mathcal{Y}$ to $\mathcal{Z}$, then it is true for all minimal geodesics from $\mathcal{Y}$ to $\mathcal{Z}$.
\end{remark}

\begin{definition}[symbolically $k$--chain separated] \label{symbolically_chain_separated_def}
Two non-intersecting hyperplanes are \textit{symbolically $k$--chain separated} if they are not symbolically $k$--chain connected and are not $k$--chain connected. 
\end{definition}

\section{Divergence in CAT(0) Cube Complexes} \label{section_div_cube_complex}
 
We now define the \textit{hyperplane divergence function}, $HDiv$, which measures the length of a shortest path between two hyperplanes which avoids a ball centered on one of the hyperplanes. Later in this section, we obtain lower bounds on the $Div$ function from lower bounds on the $HDiv$ function.

\begin{definition} \label{hyperplane_divergence_def}
Let $g$ be a minimal geodesic from $\mathcal{Y}$ to $\mathcal{Z}$ and set $p = \mathcal{Y} \cap g$. $HDiv_{g}(\mathcal{Y}, \mathcal{Z})(r)$ is the length of a shortest path from $\mathcal{Y}$ to $\mathcal{Z}$ which avoids the ball $B_p(r)$.
\end{definition}

\begin{remark}
If $X$ is one-ended then $HDiv_g(\mathcal{Y}, \mathcal{Z})(r)$ always takes finite values. Additionally, if $\mathcal{Y}$ and $\mathcal{Z}$ are $k$--separated and $g_1$, $g_2$ are different minimal geodesics from $\mathcal{Y}$ to $\mathcal{Z}$, then by Lemma \ref{strongly_separated_lemma} and Lemma \ref{k_separated_lemma} we have that $HDiv_{g_1}(\mathcal{Y}, \mathcal{Z})(r - k) \le HDiv_{g_2}(\mathcal{Y}, \mathcal{Z})(r) \le HDiv_{g_1}(\mathcal{Y}, \mathcal{Z})(r + k)$. Hence, up to the usual equivalence on divergence functions, for $k$--separated hyperplanes it is often not relevant which minimal geodesic is used. 
\end{remark}

This section is devoted to proving the following two theorems which provide a connection between the hyperplane separation properties defined in the previous section and divergence in $X$. Theorem \ref{hdiv_bounds_theorem} gives bounds on $HDiv(\mathcal{Y}, \mathcal{Z})$, and Theorem \ref{div_bounds_theorem} gives bounds on $Div(X)$.   

\begin{theorem} \label{hdiv_bounds_theorem}
The following are true: 
\begin{enumerate}
\item Suppose $X$ is finite-dimensional and locally compact. $\mathcal{Y}$ and $\mathcal{Z}$ are $\infty$-connected if and only if $HDiv(\mathcal{Y}, \mathcal{Z})$ is constant. \label{hdiv_thm_constant}
\item If $\mathcal{Y}$ and $\mathcal{Z}$ are $k$--separated, then $HDiv(\mathcal{Y}, \mathcal{Z})$ is at least linear. \label{hdiv_thm_linear2}
\item If $\mathcal{Y}$ and $\mathcal{Z}$ are $k$--chain separated and $X$ is finite-dimensional, then $HDiv(\mathcal{Y}, \mathcal{Z})(R) \succeq \frac{1}{2} R\log_2(\log_2{R}) $.\label{hdiv_thm_super_linear}

\end{enumerate}
\end{theorem}

\begin{theorem} \label{div_bounds_theorem}
Suppose $X$ is essential, locally compact and with cocompact automorphism group. 
\begin{enumerate}[resume]
\item If $\mathcal{Y}$ and $\mathcal{Z}$ are $k$--separated, then $Div(X)$ is bounded below by a quadratic function. \label{div_thm_quad}
\item If $\mathcal{Y}$ and $\mathcal{Z}$ are $k$--chain separated, then $Div(X) \succeq \frac{1}{2} R^2 \log_2(\log_2(R))$. \label{div_thm_super_quad}
\item Suppose $X$ has $k$--alternating geodesics (Definition \ref{k_alternating_geodesics}). If $\mathcal{Y}$ and $\mathcal{Z}$ are symbolically $k$--chain separated then $HDiv(\mathcal{Y}, \mathcal{Z})$ is bounded below by a quadratic function and $Div(X)$ is bounded below by a cubic function. \label{div_thm_cubic}

\end{enumerate}
\end{theorem}

\begin{definition} \label{k_alternating_geodesics}
$X$ has \textit{$k$--alternating geodesics}, if there exists a constant $M$ so that every geodesic of length $M$ in $X$ crosses a set of hyperplanes $\{\mathcal{H}_1, ..., \mathcal{H}_k \}$ such that $\mathcal{H}_i$ and $\mathcal{H}_{i+1}$ are of non--intersecting type for all $i < k$. 
\end{definition}

\begin{lemma} \label{infinitely_connected_div_lemma}
Suppose $X$ is finite-dimensional, locally compact and that $\mathcal{Y}$ and $\mathcal{Z}$ are $\infty$-connected. There exists a constant $c$ such that for all $r$ and choice of geodesic $g$, $HDiv_g(\mathcal{Y}, \mathcal{Z}) (r) = c$. 
\end{lemma}
\begin{proof}
Fix a hyperplane $H_1$ intersecting both $\mathcal{Y}$ and $\mathcal{Z}$, and let $\mathcal{H}_2$ be another hyperplane intersecting both $\mathcal{Y}$ and $\mathcal{Z}$ a distance at least $r$ from $\mathcal{H}_1$. Let $D$ be a combed disk diagram supported by $\{ \mathcal{Y}, \mathcal{H}_1, \mathcal{Z}, \mathcal{H}_2 \}$. Every dual curve to $\mathcal{H}_1$ must intersect $\mathcal{H}_2$ and every dual curve to $Y$ must intersect $Z$. Hence, $D$ is an Euclidean strip of dimension $r \times d(\mathcal{Y}, \mathcal{Z})$ and the lemma follows. 
\end{proof}

\begin{lemma} \label{k_separated_div_lemma}
Suppose $\mathcal{Y}$ and $\mathcal{Z}$ are $k$--separated. There exists a constant $c$ such that $HDiv(\mathcal{Y}, \mathcal{Z}) (r) \ge r + c$. 
\end{lemma}

\begin{proof}
Let $g$ be any minimal geodesic from $\mathcal{Y}$ to $\mathcal{Z}$ and $p = g \cap \mathcal{Y}$. Let $\alpha$ be a path from $\mathcal{Y}$ to $\mathcal{Z}$ which avoids the ball $B_p(r)$. Let $D$ be a combed disk diagram supported by $\{ \mathcal{Y}, \alpha, \mathcal{Z}, g^{-1} \}$. At most $d = |g|$ dual curves to $\mathcal{Y}$ can intersect $g$ and at most $k$ dual curves to $\mathcal{Y}$ can intersect $\mathcal{Z}$. Hence, at least $r - d - k$ dual curves to $\mathcal{Y}$ intersect $\alpha$. Therefore, $|\alpha| \ge r - d - k$. 
\end{proof}

\begin{lemma} \label{chain_separated_div_lemma}
Suppose $X$ is finite-dimensional and $\mathcal{Y}$, $\mathcal{Z}$ are $k$--chain separated. Set $d = d_{X}(\mathcal{Y}, \mathcal{Z})$.  There exists a constant $R_0(d, k)$ such that for $R > R_0$, $HDiv(\mathcal{Y}, \mathcal{Z}) \ge \frac{1}{2} R\log_2(\log_2{R})$. 
\end{lemma}

\begin{figure}[h]
\centering
\begin{overpic}[scale=.7]{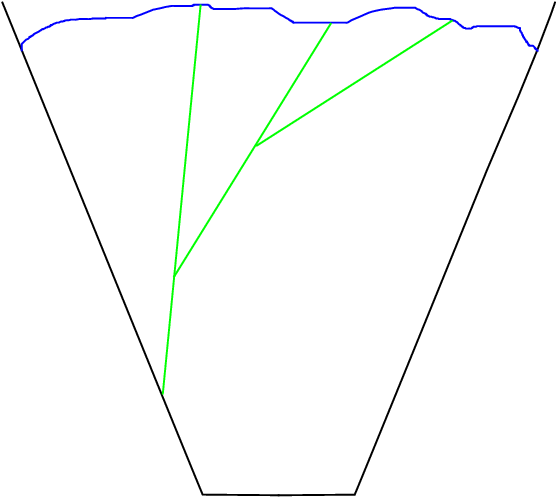}
\put(0,35){$\mathcal{Y} = \mathcal{H}_0$}
\put(80, 35){$\mathcal{Z}$}
\put(50,1){$g$}
\put(25, 70){$\mathcal{H}_1$}
\put(45, 75){$\mathcal{H}_2$}
\put(70, 75){$\mathcal{H}_3$}
\put(10, 87){$\alpha$}
\put(50, 30){$D_3$}
\put(85,80){$\alpha_3$ }

\end{overpic}
\caption{Graphic for the proof Lemma \ref{chain_separated_div_lemma}. The disk diagram $D_3$ is supported by the hyperplane-path sequence $\{ \mathcal{H}_0,   \mathcal{H}_1, \mathcal{H}_2, \mathcal{H}_3, \alpha_3, \mathcal{Z}, g^{-1} \}$. }
\label{fig_k_separated_lemma}
\end{figure}
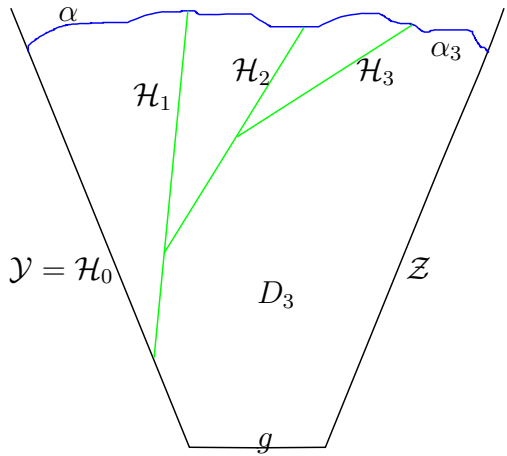

\begin{proof}
By Lemma \ref{ramsey_lemma}, let $K>0$ be the constant, only depending on $k$ and $X$, such that a geodesic of length $K$ in $X$ must intersect at least $k$ pairwise non-intersecting hyperplanes. Let $g$ be a minimal geodesic from $\mathcal{Y}$ to $\mathcal{Z}$ and set $p = g \cap Y$. Fix $R>0$ and let $r =  \log_2(\log_2{R})$. Let $\alpha$ be a combinatorial path from $\mathcal{Y}$ to $\mathcal{Z}$ which avoids the ball $B_p(R)$.

Set $\mathcal{H}_0 = \mathcal{Y}$. Let $D_0$ be a combed disk diagram with boundary supported by the hyperplane-path sequence $\bar{A}_0 = \{\mathcal{H}_0 , \alpha, \mathcal{Z}, g^{-1} \}$ and with boundary path $\bar{P}_0 = \{H_0, \alpha, Z, g^{-1} \}$. Orient $H_0$ from $g$ to $\alpha$. For $i \in \mathbb{Z}_{\ge0}$ set $c_i = 2^{i+1} r(K+d)$ where $d = |g|$.    

For $n \le r$, define inductively $H_n$ as the $c_{n-1}$'th dual curve to $H_{n-1}$ in $D_{n-1}$ and assume $H_n$ intersects $\alpha$. Orient $H_n$ from $H_{n-1}$ to $\alpha$. Define $\alpha_n$ as the subpath of $\alpha$ from $H_n$ to $Z$ and  $\beta_n$ as the subpath of $\alpha$ from $\mathcal{H}_{n-1}$ to $\mathcal{H}_{n}$. Define $D_n$ as the combed diagram with boundary supported by 

\[ \bar{A}_n = \{\mathcal{H}_0, \mathcal{H}_1, ..., \mathcal{H}_n, \alpha_n, \mathcal{Z}, g^{-1} \} \]

and with boundary path:

\[ \bar{P}_n = \{H_0, H_1, ..., H_n, \alpha_n, Z, g^{-1} \} \]

 obtained from $D_{n-1}$ by Lemma \ref{new_combed_diagram_lemma}. For $j \le n$, define the paths in $D_n$: 

\[ T_j = H_{j+1} * H_{j+2} * ... * H_n * \alpha_n \mbox{ (set } T_n= \alpha_n \mbox{)} \]
\[ B_j = Z^{-1} * g^{-1} * H_0 * H_1 * ... * H_{j-1} \]
\[ B_j' = g^{-1} * H_0 * H_1 * ... * H_{j-1}  \]

We will show through induction the following are true for all $n \le r$: 

\begin{description}
\item[A.] $D_{n}$ is well defined. In particular, $H_n$ intersects $\alpha$. 
\item[B.] For $j < n$, $rK$ dual curves to $H_j$ in $D_{n}$, intersect $T_j$. 
\item[C.] For $n >0$, $|\beta_n| \ge R - c_{n}$ 
\end{description}

Given the diagram $D_{n-1}$, in order to define $D_{n}$ we must first show $H_{n-1}$ has at least $c_{n-1}$ dual curves emanating from it in $D_{n-1}$. Since $D_n$ only needs to be defined for $n \le r$, we can do this by showing $|H_r| > 0$ in $D_r$. Because $\alpha$ avoids the $R$-ball about $p$, we have:

\[|H_{r}| \ge R - \sum_{j=0}^{r-1}c_j \]
\[ = R - \sum_{j=0}^{r-1}2^{j+1}r(K+d) = R - 2r(K+d)\sum_{j=0}^{r-1}2^{j}  \]

Using the formula for a geometric series: 

\[|H_{r}| \ge R - 2r(K+d)(2^r - 1)\]
\[ = R - \Big( 2 (K + d) \log_2( \log_2{R}) \Big) \Big( 2 \log_2{R}-1 \Big)\]

There then exists a constant $m_1(k, d)$, such that for $R > m_1$, we have that $|H_r| \ge 0$. Thus when $R > m_1$, $H_n \ge c_n$ for $n < r$.  

We now turn to the base case, $n=0$. Hypothesis A. follows immediately. Let $Q$ denote the set of dual curves to $H_0$. At most $d$ of these can intersect $g$.  Additionally, we cannot have $K$ dual curves in $Q$ intersect $Z$. For if they did, this would imply $\mathcal{Y}$ and $\mathcal{Z}$ are $k$--chain connected, a contradiction. We then have that $R - K - d$ dual curves to $H_0$ intersect $\alpha = T_0 \subset D_0$. 

The following inequalities imply that $R - K - d > rK$:

\[ R > c_0  = 2r(K+d)> rK + K + d \]

Hypothesis B. is then true, settling the base case. 

For the general case, assume $n + 1 \le r$ and that hypotheses A, B and C are satisfied for any $n' < n + 1$. Let $Q = \{Q_1, Q_2, ..., Q_m\}$ be the set of dual curves in $D_n$ emanating from $H_{n}$ ordered by the orientation on $H_{n}$. It follows at most $C = \sum_{j=0}^{n-2}c_j + d$ dual curves in $Q$ can intersect $B_{n}'$ (the sum does not go to $n-1$ since the diagram is combed). Using the formula for a geometric series we have: 

\[ C  =  c_{n-1} - 2r(K+d) + d \le c_{n-1} \]

Additionally, we cannot have $K$ curves in $Q$ intersect $Z$. For then, there is a subset of $k$ of these dual curves, $S_1 = \{H_{1}^1, ..., H_{k}^1 \} \subset Q$, corresponding to pairwise non-intersecting hyperplanes, which intersect $Z$. By induction hypothesis B, $k$ dual curves to $H_{n-1}$, $S_2 = \{H_{1}^2, H_{2}^2, ..., H_{k}^2 \}$, corresponding to pairwise non-intersecting hyperplanes, intersect every curve in $S_1$. Now, $H_{k}^1 * H_{k}^2$ is a path from $H_{n-1}$ to $Z$. By the induction hypothesis B and the pigeonhole principle, $k$ dual curves emanating from $H_{n-2}$, $S_3 = \{H_{1}^3, ..., H_{k}^3 \}$, intersect either $H_{k}^1$ or $H_{k}^2$. Hence, every curve in $S_3$ intersects every curve in $S_1$ or $S_2$.  Proceeding this way we can show $\mathcal{Y}$ is $k$--chain connected to $\mathcal{Z}$, a contradiction. 

It follows for $j \ge c_{n-1} + K + d \ge C + K + d$, $Q_j$ must intersect $\alpha$. In particular, $H_{n+1} = Q_{c_{n}}$ must intersect $\alpha$. Hence, using Lemma \ref{new_combed_diagram_lemma} we can define $D_{n+1}$, proving hypothesis A. 

Note that for $j$ such that, $c_{n-1} + K + d \le j \le c_n$, $Q_j$ must intersect $\alpha$. A direct calculation gives that there are at least $rK$ of such curves. Because of this, and the boundary combinatorics preservation property of Lemma \ref{new_combed_diagram_lemma}, hypothesis B is satisfied in $D_{n+1}$. 

We are left to prove hypothesis C. Note that for $j > c_n$, $Q_j$ intersects $\beta_{n+1}$ in $D_n$. By using the fact that $\alpha$ avoids $B_p(R)$ we have:  

\[|\beta_{n+1}| \ge R - \sum_{j=0}^{n}c_j = R - 2r(K+d)\sum_{j=0}^n2^j \]
\[ = R - 2r(K+d)(2^{n+1} - 1) = R - c_{n+1} + 2r(K+d) \]
\[ > R - c_{n+1} \]

This proves induction hypothesis C and completes the induction.

We have thus divided $\alpha$ into a set of disjoint subpaths $\{\beta_i \}$ for each of which we have a lower bound. This allows us to compute a lower bound for the length of $\alpha$: 

\[ |\alpha| \ge \sum_{i=1}^{r} |\beta_{i}| = \sum_{i=1}^{r}{ (R - c_{i} )}  \]
\[ = rR - \sum_{i=0}^{r - 1}c_{i+1} \]
\[ = rR - 4r(K+d)\sum_{i=0}^{r - 1}2^{i} \]
\[ = rR - 4r(K+d)(2^{r} - 1) \]
\[ = R \log_2(\log_2(R)) - 4(K+d)\log_2(\log_2{R}) \Big( \log_2{R} - 1  \Big) \]

There then exists a constant $m_2(k,d)$ such that for $R > m_2$

\[ |\alpha| \ge \frac{R \log_2(\log_2(R))}{2} \] 
\end{proof}

We are now in a position to prove Theorem \ref{hdiv_bounds_theorem}.

\begin{proof}[Proof of Theorem \ref{hdiv_bounds_theorem} ]
The theorem follows from the above three lemmas.

\end{proof}

The proof of the following lemma is similar to that of Lemma \ref{chain_separated_div_lemma}. However, to get the quadratic bound on the $HDiv$ function the proof requires a different counting technique. 

\begin{lemma} \label{symbolically_separated_div_lemma}
Suppose $X$ is essential, locally compact, has $k$--alternating geodesics and that Aut($X$) acts cocompactly.  Let $\mathcal{Y}$ and $\mathcal{Z}$ be symbolically $k$--chain separated. Set $d = d_X( \mathcal{Y}, \mathcal{Z})$. There exists a constant $R_0(d,k)$ such that for $R > R_0$, $HDiv(\mathcal{Y}, \mathcal{Z})$ is bounded below by a quadratic function.  
\end{lemma}

\begin{proof}
Fix $R > 0$. Let $g$ be a minimal geodesic from $\mathcal{Y}$ to $\mathcal{Z}$, $p = g \cap \mathcal{Y}$, and $\alpha$ a $B_p(R)$ avoidant combinatorial path from $\mathcal{Y}$ to $\mathcal{Z}$. Let $M$ be the $k$--alternating constant from Definition \ref{k_alternating_geodesics}. Set $c_1 = M + d$.  Let $c_2$ be the number of different hyperplane types in $X$ (this is finite since $Aut(X)$ acts cocompactly). Set $c = c_1(c_2)^k + 2d$. Set $r = \frac{R}{6c}$ and set $\mathcal{H}_0 = \mathcal{Y}$. Let $D_0$ be a combed diagram with boundary supported by $\bar{A}_0 = \{ \mathcal{Y}, \alpha, \mathcal{Z}, g^{-1} \}$ and with boundary path $\bar{P}_0 = \{H_0, \alpha, Z, g^{-1} \}$. Orient $H_0$ from $g$ to $\alpha$. 

Assume we have a combed disk diagram $D_n$ supported by 

\[ \bar{A}_n = \{\mathcal{H}_0, \mathcal{H}_1, ..., \mathcal{H}_n, \alpha_n, \mathcal{Z}, g^{-1} \} \]

and with boundary path:

\[ \bar{P}_n = \{H_0, H_1, ..., H_n, \alpha_n, Z, g^{-1} \} \]

where $\alpha_n$ is a subpath of $\alpha$ from $\mathcal{H}_n$ to $Z$. Assume $H_n$ is oriented from $H_{n-1}$ to $\alpha_n$. Define $\beta_n$ as the subpath of $\alpha$ from $\mathcal{H}_{n-1}$ to $\mathcal{H}_{n}$.  

Let $Q = \{Q_1, Q_2, ..., Q_l \}$ be the set of dual curves to $H_n$ labeled sequentially by the orientation on $H_n$. Define $\tau_n$ to be the largest integer such that $Q_{\tau_n}$ does not intersect $\alpha$. Set $m_n = \tau_n + c$, and let $H_{n+1} = Q_{m_n}$.  For $j \le n$, define the paths in $D_n$: 

\[ T_j = H_{j+1} * H_{j+2} * ... * H_n * \alpha_n \mbox{ (set } T_n= \alpha_n \mbox{)} \]
\[ B_j = Z^{-1} * g^{-1} * H_0 * H_1 * ... * H_{j-1} \]
\[ B_j' = g^{-1} * H_0 * H_1 * ... * H_{j-1} \subset B_j \]

We assume by induction the following are true for $n < r$: 

\begin{description}
\item[A.] $D_{n}$ is well defined.
\item[B.] In $D_n$, for $j < n$, exactly $c$ dual curves to $H_{j}$ intersect $T_j$ and exactly $\tau_j$ dual curves to $H_j$ intersect $B_j$. 
\item[C.] For $n>0$, $|\beta_{n}| \ge |H_{n-1}| - \tau_{n-1} - c$ 
\end{description}

For the case when $n=0$, \textbf{A} follows immediately and \textbf{B}, \textbf{C} are trivial. We now turn to the case $n=1$. Note that in $D_0$, at most $d$ dual curves can intersect $g$ and $M$ dual curves can intersect $Z$ (since $\mathcal{Y}$ and $\mathcal{Z}$ are symbolically $k$--chain separated). Hence, $\tau_0 \le d + M$, and $Q_{m_0}$ does intersect $\alpha$. Therefore, by using Lemma \ref{new_combed_diagram_lemma}, we can define $D_1$. It is also clear that \textbf{B} holds in $D_1$. Furthermore, in $D_0$, it is clear that for $j > m_0$, $Q_j$ intersects $B_1$. Hence, \textbf{C} is true as well. This shows the base cases $n=0$ and $n=1$ are true. 

We now turn to the general case. Suppose the lemma is true for all integers $n'$ such that $n' \le n < r$. Consider the diagram $D_n$, and let $Q= \{ Q_1, ..., Q_l \}$ be the set of dual curves to $H_{n}$ labeled sequentially by the orientation on $H_{n}$. 

\begin{subclaim} We will first show that we cannot have $c_1$ curves in $Q$ intersect $Z$. 
\end{subclaim}
\begin{proof}
We say two hyperplanes are almost symbolically $k$-chain connected, if they satisfy every condition of Definition \ref{symbolically_chain_connected_def} except maybe condition V. Suppose for a contradiction the claim is not true. It follows that $\mathcal{H}_n$ is almost symbolically $k$--chain connected to $\mathcal{Z}$. Assume for some $i \le n$, $\mathcal{H}_i$ and $\mathcal{Z}$ are almost symbolically $k$--chain connected by sequences: 

\[ S_1 = \{ \mathcal{P}_{1}^{1}, \mathcal{P}_{2}^{1}, ..., \mathcal{P}_{k}^{1} \} \]
\[ S_2 = \{ \mathcal{P}_{1}^{2}, \mathcal{P}_{2}^{2}, ..., \mathcal{P}_{k}^{2} \} \] 
\[ ... \]
\[ S_m = \{ \mathcal{P}_{1}^{m}, \mathcal{P}_{2}^{m}, ..., \mathcal{P}_{k}^{m} \} \]

Additionally, we want this structure to be seen in the disk diagram $D_n$. So assume for every $\mathcal{P}_i^j$, there is a corresponding dual curve $P_i^j$ in $D_n$. Also assume every dual curve corresponding to a hyperplane in $S_1$ intersects $H_i$, every dual curve corresponding to a hyperplane in $S_m$ intersects $Z$, and $P_{1}^1 * P_1^{2} * ... * P_1^{m}$ is well defined as a path in $D_n$ from $H_i$ to $Z$. 

By the induction hypothesis, there are $c$ dual cuves to $H_{i-1}$ which intersect $T_j$. Let $p$ be one such curve. Since $D_n$ is combed, $p$ cannot intersect $T_{j+1}$. Hence, $p$ must intersect $P_{1}^1 * P_1^{2} * ... * P_1^{m}$, and, consequently $p$ must intersect $P_1^{j}$ for some $j$. However, since the curves in $S_j$ are pairwise disjoint, $p$ must intersect every curve in $S_j$.  

There are only $c_2^k$ different possibilities for the tuple $Type(S_j)$. Hence, by the pigeonhole principle, there must be $k$ dual curves to $H_{n-1}$, $S = \{P_1, ..., P_k\}$, such that the corresponding hyperplanes to the sequence: 

\[ S = \{ P_{1}, P_{2}, ..., P_{k} \} \]
\[ S_j = \{ P_{1}^{1}, P_{2}^{1}, ..., P_{k}^{1} \} \]
\[ S_{j+1} = \{ P_{1}^{2}, P_{2}^{2}, ..., P_{k}^{2} \} \] 
\[ ... \]
\[ S_m = \{ P_{1}^{m}, P_{2}^{m}, ..., P_{k}^{m} \} \]

almost symbolically $k$--chain connect $\mathcal{H}_{n-1}$ to $\mathcal{Z}$. 

Proceeding in this manner, this would imply $\mathcal{Y}$ is symbolically $k$--chain connected to $\mathcal{Z}$, a contradiction. This finishes the proof of the subclaim.
\end{proof}

We next want to show that $m_n$ is well defined. By induction hypothesis \textbf{B} and the subclaim, at most $C = (n-2)c + d + c$ curves in $Q$ can intersect $B_n$. Note that $\frac{R}{6} = rc \ge nc + c \ge C + c$. Hence, for $m_n$ to be well defined, it is enough to know that $|H_n| \ge \frac{R}{6}$.

$\sum_{j=0}^{n-1}\tau_j$ is a count of how many dual curves have both endpoints on $B_n$. At most $(n-1)c$ such curves have endpoints on $H_i$ and $H_j$ for some $i < j < n$. At most $d$ such curves have endpoints on $g$. Furthermore, by the subclaim  at most $(n-1)c_1$ such curves have endpoints on  $H_i$ and $Z$ for some $i < n$. Hence, 

\[ \sum_{j=0}^{n-1}\tau_j \le  (n-1)c + d + (n-1)c_1 \le 2nc \]  

Since $\alpha$ does not intersect $B_p(R)$, we have that: 

\[ |H_n|  \ge R - \sum_{i=0}^{n-1}{|H_i|} \]
\[ = R - \sum_{i=0}^{n-1}{(\tau_i + c)} = R - nc - \sum_{i=0}^{n-1}{\tau_i} \]
\[ \ge R -3nc \]
\[ \ge R - 3rc = \frac{R}{2}  \]

Thus, $m_n$ is well defined. 

We are left to prove induction hypothesis \textbf{C}. For $j > m_n$, $Q_j$ intersects $\beta_{n+1}$. Thus, 

\[ |\beta_n| \ge |H_{n-1}| - \tau_{n-1} - c \]

This finishes the induction. 

We have broken $\alpha$ into a union of $r$ subpaths $\{\beta_j \}$ for which we have a lower bound. We can then calculate a lower bound for $\alpha$: 

\[ 
|\alpha| \ge \sum_{i=0}^{r-1}|\beta_{i+1}|  \ge  \sum_{i=0}^{r-1}|H_i| - \tau_i - c
\]

\[
 \ge \sum_{i=0}^{r-1} \Big( \frac{R}{2} - \tau_i - c\Big) 
\]

\[ 
\ge \frac{rR}{2} - \sum_{i=0}^{r-1} (\tau_i) -rc
\]

\[ 
\ge \frac{rR}{2} - 3rc = \frac{R^2}{12c} - \frac{R}{2}
\]

\end{proof}

The following theorem allows us to deduce stronger lower bounds for $Div(X)$ through the existence of just two hyperplanes with strong enough separation properties. The proof of the theorem involves constructing an infinite sequence of nested hyperplanes. This is done primarily through the machinery developed in \cite{CS}. 

\begin{theorem} \label{chaining_hyperplanes_theorem}
Let $X$ be essential, locally compact and with cocompact automorphism group. Suppose $HDiv(\mathcal{Y}, \mathcal{Z}) \succeq F(r)$ for a pair of non-intersecting hyperplanes $\mathcal{Y}$ and $\mathcal{Z}$ in $X$. It then follows that $Div(X) \succeq rF(r)$.  
\end{theorem}
\begin{proof}

Let $\mathcal{Y}^+$ and $\mathcal{Z}^+$ be half-spaces associated to $\mathcal{Y}$ and $\mathcal{Z}$ such that $\mathcal{Y}^+ \subsetneq \mathcal{Z}^+$. 

By the Double Skewering Lemma in \cite{CS}, there exists a $\gamma \in G$ so that $\gamma \mathcal{Z}^+ \subsetneq \mathcal{Y}^+$. Note that $HDiv(\gamma \mathcal{Z}, \mathcal{Z}) \ge F(r)$ since $\mathcal{Y}$ separates $\mathcal{Z}$ from $\gamma \mathcal{Z}$.  By Lemma 2.3 in \cite{CS}, $\gamma$ is hyperbolic and its axis, $l$, intersects $\mathcal{Y}$ and $\mathcal{Z}$ ($\gamma$ skewers both $\mathcal{Y}$ and $\mathcal{Z}$). 

Now, we have a chain of equally spaced pairs of hyperplanes $\{\gamma^n Z, \gamma^{n-1}Z \}$ along $l$ (isometry moves hyperplanes through $l$). Hence, the divergence of the geodesic $l$ is at least $rF(r)$. 
\end{proof}

\begin{proof}[Proof of Theorem \ref{div_bounds_theorem} ]
The  statements in Theorem \ref{div_bounds_theorem} are now an easy consequence of Lemma \ref{symbolically_separated_div_lemma}, Theorem \ref{hdiv_bounds_theorem} and Theorem \ref{chaining_hyperplanes_theorem}. 
\end{proof}

\begin{remark}
We note that Theorem \ref{chaining_hyperplanes_theorem}, Theorem \ref{div_bounds_theorem} and Theorem \ref{higher_degree_theorem} all hold under the different assumption that $X$ is essential, finite-dimensional and $Aut(X)$ has no fixed point at infinity. This is true since the Double Skewering Lemma from \cite{CS} also holds under these assumptions.
\end{remark}

\section{Higher Degree Polynomial Divergence in CAT(0) Cube Complexes} \label{section_higher_bounds}

We define when a pair of hyperplanes are \textit{degree $d$ $k$--separated}. We show the divergence of two degree $d$ $k$--separated hyperplanes is bounded below by a polynomial of degree $d$. 

\begin{definition} \label{degree_d_separated_def}
Hyperplanes $\mathcal{H}_1$ and  $\mathcal{H}_2$ are degree $1$ $k$--separated if $\mathcal{H}_1$ and $\mathcal{H}_2$ are $k$--separated. 

$\mathcal{H}_1$ and  $\mathcal{H}_2$ are degree $d$ $k$--separated if they are $k$--separated, and for either $i=1$ or $i=2$ every geodesic of length $k$ contained in $N(\mathcal{H}_i)$ intersects a pair of degree $(d-1)$ $k$--separated hyperplanes.
\end{definition}

The main theorem of this section is the following:

\begin{theorem} \label{higher_degree_div_theorem}
Suppose $X$ is finite-dimensional. If $\mathcal{Y}$ and $\mathcal{Z}$ are degree $d$ $k$--separated hyperplanes, then $Hdiv(\mathcal{Y}, \mathcal{Z})$ is bounded below by a polynomial of degree $d$. 
\end{theorem}

By combining Theorem \ref{chaining_hyperplanes_theorem} and \ref{higher_degree_div_theorem} we immediately get the following:

\begin{theorem} \label{higher_degree_theorem}
Let $X$ be an essential, locally compact CAT(0) cube complex with cocompact automorphism group. If $X$ contains a pair of degree $d$ $k$--separated hyperplanes, then $Div(X)$ is bounded below by a polynomial of degree $d+1$. 
\end{theorem}

\begin{proof}[Proof of Theorem \ref{higher_degree_div_theorem} ]

The base case, $d = 1$, follows from \ref{hdiv_thm_linear2} of Theorem \ref{hdiv_bounds_theorem}. For the general case, assume the claim is true for degree $d-1$ $k$--separated hyperplanes. Suppose $\mathcal{Y}$ and $\mathcal{Z}$ are degree $d$ $k$--separated. Let $g$ be a minimal geodesic from $\mathcal{Y}$ to $\mathcal{Z}$ and let $p = g \cap \mathcal{Y}$. Fix $R > 0$, and let $\alpha$ be a path from $\mathcal{Y}$ to $\mathcal{Z}$ which avoids the ball $B_p(R)$. Let $D$ be a combed disk diagram supported by $\{ \mathcal{Y}, \alpha, \mathcal{Z}, g^{-1} \}$. 

Orient $Y \subset D$ from $g$ to $\alpha$, and let $A = \{H_1, H_2, ..., H_n \}$ be dual curves to $Y$ sequentially ordered by the orientation on $Y$. Since $\alpha$ does not intersect $B_r(p)$, we have that $n \ge r$. Since $D$ is combed and since $\mathcal{Y}$ and $\mathcal{Z}$ are $k$--separated, it follows for $i > |g| + k$, $H_i$ intersects $\alpha$. By Definition \ref{degree_d_separated_def}, there is a subsequence $B = \{ K_1, K_2, ..., K_m \} \subset A$ such that:

\begin{enumerate}
\item $K_i$ intersects $\alpha$
\item For $i$ odd, $K_i$ and $K_{i+1}$ are degree $d-1$ $k$--separated. 
\item $m \ge \frac{(r - k - |g|)}{k}$
\item $d(K_i, p) \le ki + |g| + k$
\end{enumerate}

For $i$ odd, let $\alpha_i$ be the segment of $\alpha$ from $\mathcal{K}_i$ to $\mathcal{K}_{i+1}$. Note that $\alpha_i$ is a path from $\mathcal{K}_i$ to $\mathcal{K}_{i+1}$ which avoids the ball $B_{(r - ki - |g| - k)}(\mathcal{K}_i \cap \mathcal{Y})$. By the induction hypothesis and Lemma \ref{k_separated_lemma}, $|\alpha_i| \ge (r - ki - |g| - k)^{d-1}$. Since we have linearly many segments $\{\alpha_i\}$ whose length is bounded below by a degree $d-1$ polynomial, it follows the length of $\alpha$ is bounded below by a degree $d$ polynomial. 
\end{proof}

\section{Right-Angled Coxeter Group Divergence} \label{section_racg_divergence}

In the next two subsections, we wish to apply the theorems from previous sections to the case of right-angled Coxeter groups (RACGs for short). 

Let $\Gamma$ be the graph associated to a RACG, $W_{\Gamma}$. Let $\Gamma^c$ be the graph complement of $\Gamma$ and let $I$ be the set of isolated vertices in $\Gamma^c$. i.e., $I = \{ v \in V(\Gamma^c) ~ | ~ Link(v) = \emptyset\}$.   

$I$ forms a clique in $\Gamma$, and $\Gamma$ is the graph join of the induced subgraph corresponding to $I$ with the induced subgraph corresponding to $\Gamma - I$. Consequently, $W_{(\Gamma - I)}$ is finite index in $W_{\Gamma}$. Divergence is a quasi-isometry invariant, hence divergence results for $W_{(\Gamma - I)}$ apply to $W_{\Gamma}$. 

\textbf{We will from now on assume, without loss of generality, that $\Gamma^c$ has no isolated vertices for all RACGs considered.} The Davis complex for $W_{\Gamma}$ under this assumption is essential.

The following definition is a construction used in \cite{DT}.

\begin{definition}[$\Gamma$-complete word] \label{complete_word_def}
Given a graph $\Gamma$ which is not a join, let $w_0 = s_1...s_k$ be a word with the property that for every generator $s \in \Gamma$, there exists an $i$ such that $s_i = s$. Furthermore, $m(s_i, s_{i+1}) = \infty$ for all $1 \le i < k$ and $m(s_1, s_k) = \infty$. Since $\Gamma$ is not a join, it is always possible to define $w_0$, although $w_0$ is not unique. We call such a word a \textit{$\Gamma$-complete word} and always denote it by $w_0$. 
\end{definition}

We use the definition of a CFS graph used in \cite{BFHS} and \cite{DT} (defined below). An \textit{induced square} of a graph $\Gamma$ is an embedded 4--cycle. 
 
\begin{definition} \label{cfs_def}
Given a graph $\Gamma$, define $\square (\Gamma)$ as the graph whose vertices are induced squares of $\Gamma$. Two vertices in $\square ( \Gamma )$ are adjacent if and only if the corresponding induced squares in $\Gamma$ have two non-adjacent vertices in common. For a set of induced squares $S \subset \square ( \Gamma)$, define the \textit{support} of $S$ to be all vertices in $\Gamma$ which are contained in some square in $S$. We say $\Gamma$ is \textit{CFS} if $\square (\Gamma)$ contains a component whose support is $V(\Gamma)$. 
\end{definition}

\begin{remark}
In \cite{BFHS}, the graph join of a CFS graph with a clique graph is still CFS. With the assumption that $\Gamma^c$ has no isolated vertices, such a graph is not possible and so we omit this from the definition. We note again, however, that the RACG corresponding to a graph that is a join with a clique is commensurable with the RACG corresponding to the graph. So the results in this section still hold in full generality.  
\end{remark}

\subsection{Characterization of Linear Divergence}

The authors of \cite{BFHS} characterize which right-angled Coxeter groups exhibit linear divergence and the triangle-free case is done in \cite{DT}. For completeness and as a warm up for the quadratic case, we provide another proof here of this characterization.

\begin{theorem}
If $\Gamma$ is a join then $Div(W_{\Gamma})$ is linear. Otherwise, $Div(W_{\Gamma})$ is at least quadratic.  
\end{theorem}
\begin{proof}
Suppose $\Gamma$ is a join. It follows $W_{\Gamma} = W_{\Gamma_1} \times W_{\Gamma_2}$. By the assumption that $\Gamma^c$ has no isolated vertices, both $W_{\Gamma_1}$ and $W_{\Gamma_2}$ are infinite. Hence, $Div(W_{\Gamma})$ is linear. 

Now suppose $\Gamma$ is not a join. Let $w_0 = s_1s_2...s_k$ be a $\Gamma$--complete word. Let $\mathcal{Y}$ be the hyperplane dual to the letter $s_1$ in $w_0$ and $\mathcal{Z}$ the hyperplane dual to the letter $s_k$ in $w_0$ in the Davis complex of $W_{\Gamma}$. Since, $m(s_1, s_k) = \infty$, it follows $\mathcal{Y}$ and $\mathcal{Z}$ do not intersect. Similarly, any hyperplane dual to the letter $s_j$ in $w_0$ for $1 < j < k$, does not intersect $\mathcal{Y}$ or $\mathcal{Z}$. 

We will show $\mathcal{Y}$ and $\mathcal{Z}$ are strongly separated. Suppose, for a contradiction, some hyperplane $\mathcal{H}$ intersects both $\mathcal{Y}$ and $\mathcal{Z}$. Let $\mathcal{H}$ be of type $s \in V(\Gamma)$. It follows that for every $j$ such that $1 \le j \le k$, the hyperplane through the letter $s_j$ in $w_0$ intersects $\mathcal{H}$. Hence, for every $t \in \Gamma$, $m(t,s) = 2$. But this implies that $s$ is isolated in $\Gamma^c$, a contradiction. 

Since $\mathcal{Y}$ and $\mathcal{Z}$ are strongly separated, by Theorem \ref{div_bounds_theorem}, $Div(W_{\Gamma})$ is at least quadratic. 

\end{proof} 

\subsection{Characterization of Quadratic Divergence}

We use results from Section \ref{section_div_cube_complex} to characterize quadratic divergence in RACGs and show there is a gap between quadratic and cubic divergence in RACGs.  

\begin{theorem} \label{racg_quadratic_div_theorem}
Suppose $\Gamma$ is not CFS and is not a join, then $W_\Gamma$ has divergence greater or equal to a cubic polynomial.
\end{theorem} 

The proof of the theorem will be given at the end of this subsection. We state the following corollaries which immediately follow. 

\begin{corollary}
$W_{\Gamma}$ has quadratic divergence if and only if $\Gamma$ is CFS and is not a join. 
\end{corollary}

\begin{proof}
If $\Gamma$ is CFS and is not a join, it follows from \cite{BFHS} that it has quadratic divergence. The other direction follows from Theorem \ref{racg_quadratic_div_theorem}. 
\end{proof}

\begin{corollary} \label{thick_order_3_cor}
If $W_\Gamma$ is strongly thick of order 2, then $W_{\Gamma}$ has cubic divergence. 
\end{corollary}

\begin{proof}
By \cite{BFHS}, $W_{\Gamma}$ has at most cubic divergence. Hence, by Theorem \ref{racg_quadratic_div_theorem}, $W_{\Gamma}$ has exactly cubic divergence. 
\end{proof}

The following lemma guarantees the existence of symbolically $k$--chain separated hyperplanes when $\Gamma$ is not CFS. 

\begin{lemma} \label{cfs_condition_lemma}
Let $M$ be the maximal clique size in $\Gamma$. Let $w_0 = s_1s_2...s_k$ be a $\Gamma$-complete word and consider its image in the Davis complex $X$. Let $\mathcal{Y}$ be the hyperplane dual to $w_0$ which intersects $s_1$ and $\mathcal{Z}$ the hyperplane dual to $w_0$ intersecting $s_k$. If $\mathcal{Y}$ and $\mathcal{Z}$ are symbolically $2$-chain connected then $\Gamma$ is CFS. 
\end{lemma}

\begin{proof}
Assume $\mathcal{Y}$ and $\mathcal{Z}$ are symbolically $2$-chain connected by sequences: 

\[ S_1 = \{ \mathcal{H}^{1}, \mathcal{K}^{1} \} \]
\[ S_2 = \{ \mathcal{H}^{2}, \mathcal{K}^{2} \} \] 
\[ ... \]
\[ S_m = \{ \mathcal{H}^{m}, \mathcal{K}^{m} \} \]

Let $a = \{ a_{1}, a_{2}, ..., a_{m} \}$ and $b = \{ b_{1}, b_{2}, ..., b_{m} \}$ be the letters in $\Gamma$ corresponding respectively to the hyperplanes $\{ \mathcal{H}^1, ..., \mathcal{H}^m \}$ and $\{ \mathcal{K}_1, ..., \mathcal{K}^m\}$. It follows $a \cup b$ forms a CFS subgraph, $\Delta$, of $\Gamma$. 

Any hyperplane intersecting $w_0$ cannot intersect $\mathcal{Y}$ or $\mathcal{Z}$. Thus any such hyperplane separates $\mathcal{Y}$ from $\mathcal{Z}$. Consequently each hyperplane intersecting $w_0$ must intersect $\mathcal{H}_i$ and $\mathcal{K}_i$ for some $i$. Let $L(w_0)$ be the set of generators in the word $w_0$, namely $L(w_0) =  \{s_1, s_2, ..., s_k\}$. It follows that for each $s \in L(w_0)$, there exists a $j$ such that $s$ commutes with both $h_j, k_j \in \Delta$.

Given $s \in L(w_0)$, assume $s \notin a \cup b$ as a vertex of $\Gamma$, and assume $s$ does not commute with every generator in $a \cup b$. Let $t \in a \cup b$ be such that $m(s, t) = \infty$, $t \in \{a_r, b_r\}$ for some $r$, and $m(s, a_j) = m(s, b_{j}) = 2$ for some $j$ with $|r-j| =1$. This is possible by the above paragraph. It follows $\{s, t, a_j, b_j \}$ forms an induced square which shares two non-adjacent vertices with a square in $\Delta$. 

On the other hand, suppose $s_i \in L(w_0)$ commutes with every generator in $A \cup B$. $s_{i+1}$ (if $i = k$ set $s_{i+1} = s_1$) commutes with $a_j$ and $b_j$ for some $j$. We then have that $\{s_i, s_{i+1}, a_j, b_j \}$ forms an induced square which which shares two non-adjacent vertices with a square in $\Delta$. 

We have thus shown every generator in $L(w_0)$ is either contained in $\Delta$ or contained in an induced square C which shares two non-adjacent vertices with a square in $\Delta$. Since $L(w_0)$ contains every generator in $\Gamma$, we have shown that $\Gamma$ is CFS.
\end{proof}

\begin{lemma} \label{racg_alternating_lemma}
The Davis complex for $W_{\Gamma}$ has $2$-alternating geodesics. 
\end{lemma}
\begin{proof}
Choose $M$ to be one larger than the maximal clique size in $\Gamma$. Let $g = s_1s_2...s_M$ be a geodesic of length $M$ with $s_i \in \Gamma$. By Tits' solution to the word problem (see \cite{Dav}), it follows that $m(s_i, s_j) = \infty$ for some $1 \le i < j \le k$. It follows the hyperplane intersecting $s_i$ and the hyperplane intersecting $s_j$ are of non-intersecting type. 
\end{proof}

\begin{proof} [Proof of Theorem \ref{racg_quadratic_div_theorem}]
Theorem \ref{racg_quadratic_div_theorem} now follows from the above two lemmas and Theorem \ref{div_bounds_theorem}. 
\end{proof}

\subsection{Higher Degree Polynomial Divergence in RACGs}

In this section, we apply results from Section \ref{section_higher_bounds} to give graph-theoretic criteria which imply lower bounds on the divergence of a RACG. Together with the machinery of thickness (see \cite{BHSC} and \cite[Corollary 6.3.1]{Lev2}) which provides upper bounds on divergence, these results allow one to compute the exact divergence of many RACGs. 

\begin{definition}
Given distinct vertices $s, t \in \Gamma$, $(s, t)$ is a \textit{non-commuting pair} if $s$ is not adjacent to $t$ in $\Gamma$. 
\end{definition}

\begin{definition} \label{rank_n_pair}
A non-commuting pair $(s, t)$ is rank 1 if $s, t$ are not contained in some induced square of $\Gamma$. Additionally $(s, t)$ are rank $n$ if either every non-commuting pair $(s_1, s_2)$, with $s_1, s_2 \in Link(s)$, is rank $n-1$ or every non-commuting pair $(t_1, t_2)$, with $t_1, t_2 \in Link(t)$, is rank $n-1$. 
\end{definition}

\begin{figure}[h]\label{fig_rank_n_example}
\centering
\begin{overpic}[scale=.3]{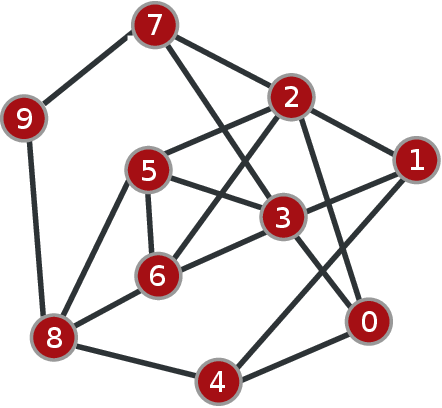}

\end{overpic}
\caption{The non-commuting pairs (4,6), (4,5), (4,9), (5,9) and (6,9) for the example graph above are rank 1. It then follows that the non-commuting pair (7,8) is rank 2. Taking this one step further, we see that the non-commuting pair (9,0) is rank 3. By Theorem \ref{higher_degree_div_racg}, the RACG associated to the above graph has divergence bound below by a polynomial of degree 4. Furthermore,  it can easily be checked using techniques from \cite{BHSC} that this RACG is thick of order 3 and so the divergence of this group is exactly a quartic polynomial.}
\end{figure}

\begin{theorem} \label{higher_degree_div_racg}
Suppose $\Gamma$ contains a rank $n$ pair $(s, t)$, then $Div(W_{\Gamma})$ is bounded below by a polynomial of degree $n+1$. 
\end{theorem}
\begin{proof}
Let $M$ be the maximal clique size in $\Gamma$. We claim that hyperplanes of type $s$ and $t$ must be degree $n$ $M$-separated, in the sense of Definition \ref{degree_d_separated_def}. If this claim is shown, the theorem follows from Theorem \ref{higher_degree_theorem}. 

We first prove the base case when $n = 1$. Suppose, for a contradiction, $\mathcal{Y}$ and $\mathcal{Z}$ are of type $s$ and $t$ respectively and that $M+1$ hyperplanes intersect both $\mathcal{Y}$ and $\mathcal{Z}$. It follows from Lemma \ref{racg_alternating_lemma} that two such hyperplanes, $\mathcal{H}$ and $\mathcal{H}'$ are respectively of type $a, b \in \Gamma$ where $(a,b)$ is a non-commuting pair. However, it then follows $\{s, a, b, t \}$ is an induced square in $\Gamma$, contradicting $(s,t)$ being rank 1. 

For the general case, suppose $(s, t)$ are rank $n$ and $\mathcal{Y}$ and $\mathcal{Z}$ are hyperplanes of type $s$ and $t$ respectively. Without loss of generality, assume every non-commuting pair $(s_1, s_2)$, with $s_1, s_2 \in Link(s)$,  are rank $n-1$. By the induction hypothesis, hyperplanes of type $s_1$ and type $s_2$ are degree $n-1$ $M$-separated. 

Consider any geodesic, $g \subset N(\mathcal{Y})$, of length $M+1$. By Lemma \ref{racg_alternating_lemma}, $g$ crosses two hyperplanes of non-commuting type, say of type $s_1$ and type $s_2$. By the above paragraph, $(s_1, s_2)$ must be degree $n-1$ $M-1$-separated. The claim then follows.
\end{proof}

\begin{remark}\label{cubic_counterexample_rmk} It is not true that the largest rank of a pair of vertices of a graph determines the corresponding RACG's divergence. The graph, $\Gamma$, in Figure \ref{cubic_counterexample_fig} is not a join and is not CFS. Therefore, the divergence of $W_{\Gamma}$ is at least cubic by Theorem \ref{racg_quadratic_div_theorem}. In fact, by applying \cite[Corollary 6.3.1]{Lev2} to obtain an upper bound, the divergence is determined to be exactly cubic. Furthermore, every non-adjacent pair of vertices in $\Gamma$ is either rank $0$ or rank $1$. It follows we can only obtain a quadratic lower bound on the divergence of $W_{\Gamma}$ through Theorem \ref{higher_degree_div_racg}.
\end{remark}

\begin{figure}[h]
	\centering
	\begin{overpic}[scale=.3]{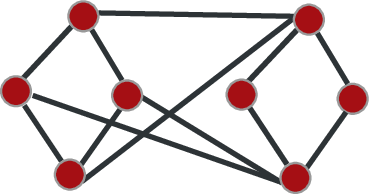}

	\end{overpic}
	\caption{A graph that is not a join, is not CFS and only contains rank $0$ and rank $1$ pairs of vertices.} \label{cubic_counterexample_fig}
\end{figure}

\section{Coxeter Groups} \label{section_coxeter_groups}

This section explores lower bounds for divergence in Coxeter groups (not necessarily right-angled). We use similar arguments to those used in Section \ref{section_higher_bounds}. We do not make use of a cube complex in this section. Instead, we use the construction of bands in Van-Kampen diagrams which behave similarly to dual curves in CAT(0) cube complex disk diagrams. We refer the reader to \cite[Chapter 4]{Olshanskii} for a background on Van-Kampen diagrams and to \cite{Bahls} for their application to Coxeter groups.

A characterization of thick Coxeter groups is given in \cite[Proposition A.2]{BHSC} by a class of edge-labelled graphs that can be constructed by an inductive procedure. Furthermore, the authors' proof of this proposition provides an upper bound on thickness, and hence divergence, at each step of the inductive construction. One can then carefully apply the results in this section, together with the work in \cite{BHSC}, and obtain the exact divergence for a large class of Coxeter groups.

In this section, we assume all Coxeter diagrams have at least one edge. Otherwise, $W_{\Gamma}$ is virtually trivial or virtually free and exhibits either trivial or infinite divergence. 

\subsection{Higher Degree Polynomial Divergence}

\begin{definition} \label{locally_even}
Let $\Gamma$ be a labeled graph. The vertex $v \in V(\Gamma)$ is \textit{$r$-locally triangle-free} if for all $u \in V(\Gamma)$, such that $d_{\Gamma}(u, v) < r$, $u$ is not in a triangle. We say $v$ is \textit{$r$-locally even} if for all $u \in V(\Gamma)$, such that $d_{\Gamma}(u, v) < r$, each edge adjacent to $u$ is even labeled or not labeled. 
\end{definition}

\begin{definition}
For $\Gamma$ a Coxeter diagram, let $L_{\Gamma}$ be the largest integer edge label in $\Gamma$. If $\Gamma$ contains no labeled edges, set $L_{\Gamma} = 2$. 
\end{definition}

The following lemma allows us to choose boundedly spaced generators in a minimal expression for a word $w \in W_{Star(v)}$ such that these generators are not $v$ and do not sequentially coincide. 

\begin{lemma} \label{general_ordering_lemma}
Let $\Gamma$ be a Coxeter diagram. For any $v \in \Gamma$ and any minimal expression, $w = s_1s_2...s_n$, $s_i \in Star(v)$, for a word $w \in W_{Star(v)}$, there exists a subsequence $\{ s_{i_1}, s_{i_2}, ...,  s_{i_m} \} $ such that 

\begin{enumerate}

\item $s_{i_j} \neq v$ for all $j$. 
\item $i_1 \le 2$
\item $i_{j+1} - i_j \le L_\Gamma$ 
\item $s_{i_{j+1}} \neq s_{i_j}$ as vertices of $\Gamma$ 
\item $m \ge \frac{n}{L_\Gamma + 1}$. 

\end{enumerate}
\end{lemma}

\begin{proof}
Since $w$ is minimal length, there cannot be two $v$ letters appearing consecutively. Hence either the first or second letter is not $v$. Set $t_{i_1}$ to be this letter. Now note that for any letter $s \in Link(v)$ and $n > L_{\Gamma}$, we cannot have the expression $svsv...sv$ or $svsv...svs$ of length $n$ appearing in $w$ for this would contradict $w$ being reduced. Hence, there is some letter, $t_{i_2}$ not equal to $t_{i_1}$ or $v$ with $i_2 - i_1 \le L_{\Gamma}$. We can keep repeating this process, and the lemma follows.   
\end{proof}

Let $D$ be a Van-Kampen diagram for a Coxeter group. Each 2-cell in $D$ has an even number of edges along its boundary path. For a given cell and a given edge along the cell's boundary path, there is a corresponding opposite edge. Furthermore, each edge in $D$ is contained in exactly one cell if it is a boundary edge of $D$ and in exactly two cells if it is not. Two edges $e$ and $e'$ in $D$ are \textit{opposite connected} if there is a sequence of edges $e = e_1, e_2, ..., e_n = e'$ such that $e_i$ is opposite to $e_{i+1}$ in some 2-cell of $D$. A \textit{band} associated to an edge $e$ in $D$ is the set of all edges opposite connected to $e$ and cells adjacent to these edges.  

The construction of bands is utilized in \cite[Section 1.4]{Bahls}. There it is also shown that bands do not self-intersect and cannot intersect geodesics twice. 

For $u, v \in \Gamma$, an \textit{odd path} from $u$ to $v$ is a path in $\Gamma$ which only contains edges with odd labels. Let $O_v$ consist of vertices $u \in \Gamma$ for which there is an odd path from $u$ to $v$. By definition $v \in O_v$. 

Let $e$ be an edge in $D$ labeled by some $v \in \Gamma$. It is easy to check the band corresponding to $e$ only contains edges labeled by elements in $O_v$. 

Recall that Definition \ref{rank_n_pair} of a rank $n$ pair is still valid for Coxeter groups which are not right-angled. 

\begin{theorem} \label{coxeter_higher_div_thm}
Let $\Gamma$ be a Coxeter graph. Suppose $(u, v)$ is a rank $n$ pair. Without loss of generality, we assume that for all distinct $u_1, u_2 \in Link(u)$, $(u_1, u_2)$ is a rank $n-1$ pair in $\Gamma$. Further assume that $u$ is $n$-locally triangle free and $n+1$-locally even and that $v$ is $1$-locally even. It follows that the divergence of the bi-infinite geodesic $...uvuv...$ is bounded below by a polynomial of degree $n+1$. 
\end{theorem}

The following corollary is immediate: 

\begin{corollary}
Let $W_{\Gamma}$ be an even Coxeter group such that $\Gamma$ contains no triangles. If $(u,v)$ is a rank $n$ pair in $\Gamma$, then $Div(W_{\Gamma})$ is bounded below by a polynomial of degree $n+1$. 
\end{corollary}

To prove the above theorem, we will first need the following technical lemma: 

\begin{lemma} \label{general_div_lemma}
Let $(u, v)$ be as in Theorem \ref{coxeter_higher_div_thm}. Let $g \in W_{Star(u)}$ and $h \in W_{Star(v)}$ and $p \in W_{\Gamma}$ be words written in a minimal length expression. Suppose $|p| \le L_{\Gamma}$, $|g| \ge r$ and $|ph| \ge r$. Let $\alpha$ be a shortest path from $g$ to $ph$ in the Cayley graph of $W_{\Gamma}$ which does not intersect $B_{id}(r)$. It follows $|\alpha|$ is bounded below by a polynomial of degree $n$. 
\end{lemma}

\begin{proof}
The proof will follow by induction on $n$. We begin with the base case where the rank of $(u,v)$ is $n=1$. Suppose $g$ is given by the following expression in generators, $g = s_1s_2...s_l$. Note that $l \ge r$. 	Let $D$ be a Van-Kampen with boundary path $g \alpha h^{-1} p^{-1}$. 

Let $T = \{s_{i_1}, s_{i_2}, ..., s_{i_m} \}$ be a subsequence of $\{s_1, s_2, ..., s_l\}$ as in Lemma \ref{general_ordering_lemma}, and $B = \{B_1, B_2, ..., B_m \}$ bands in $D$ corresponding to each letter in $T$. Since $u$ is $2$--locally even, each band $B_j$ only contains edges labeled by $s_{i_j}$. Furthermore, since $u$ is $1$-locally triangle free, $u$ is not contained in a triangle. It follows for $i \ne j$, $B_i$ and $B_j$ do not intersect. 

At most $L_{\Gamma}$ bands can intersect $p$. Additionally, $u$ and $v$ are rank 1, and so are not in a common square of $\Gamma$. It follows for $j > L_{\Gamma}$, $B_j$ intersects $\alpha$. Hence, $|\alpha|$ is linear in $r$, proving the base case. 

Now assume the theorem is true for $n-1$ and that $(u,v)$ are of rank $n$. The proof proceeds almost the same way as the base case. Consider all the same notation as the base case. For $i$, such that $L_{\Gamma} + 1 < i < m$, let $\alpha_i$ be the segment of $\alpha$ between $B_i$ and $B_{i+1}$, and let $p_i$ be the segment of $g$ from $B_i$ to $B_{i+1}$. Let $g_i$ be the word along $B_i$ from $p_i$ to $\alpha_i$, and let $h_i$ be the word along $B_{i+1}$ from $p_i$ to $\alpha_i$. By the induction hypothesis, $|\alpha_i|$ is bounded below by a polynomial of degree $n-1$ in $r -i$. Hence, $|\alpha|$ is bounded below by a polynomial of degree $n$. 
\end{proof}

\begin{proof} [Proof of Theorem \ref{coxeter_higher_div_thm} ]
Let $\alpha$ be a $B_{id}(r)$ avoidant path from $(uv)^r$ to $(vu)^r$ in the Cayley graph of $W_{\Gamma}$. Let $D$ be a Van-Kampen diagram with boundary path $(uv)^r \alpha (vu)^{-r}$. Since $(u,v)$ is a non-commuting pair in $\Gamma$, no pair of bands emanating from the words $(uv)^r$ or from $(vu)^r$ along the boundary path of $D$ can intersect. Hence, each of these bands must intersect $\alpha$. 

Write $(uv)^r$ as $u_1v_1u_2v_2...u_rv_r$. Let $U_i, V_i$ be bands corresponding respectively to $u_i, v_i$. Let $D_i$ be the minimal connected subdiagram of $D$ which includes $U_i$ and $V_i$. Let $\alpha_i$ be the segment of $\alpha$ contained in $D_i$. By Lemma \ref{general_div_lemma}, $|\alpha_i|$ is bounded below by a polynomial of degree $n$ in $r - i$. Hence, $|\alpha|$ is bounded below by a polynomial of degree $n + 1$.  

\end{proof}

\subsection{Quadratic Divergence Lower Bound}

\begin{definition}
Given an edge labeled Coxeter graph $\Gamma$, let $\hat{\Gamma}$ be the graph resulting from collapsing odd labeled edges in $\Gamma$ to a point. For $v \in \Gamma$ we denote its image in $\hat{\Gamma}$ by $\pi(v)$.  Each vertex $\hat{v} \in \hat{\Gamma}$ is labeled by a list, $\pi^{-1}(\hat{v})$. Each edge in $\hat{\Gamma}$ is labeled by the same integer as the corresponding edge in $\Gamma$. Note that this new graph can have multiple edges between two vertices. 
\end{definition}

\begin{theorem} \label{coxeter_quadratic_thm}
Let $\Gamma$ be a Coxeter graph. If the diameter of $\hat{\Gamma}$ is larger than 2, then $W_{\Gamma}$ has at least quadratic divergence. 
\end{theorem}

\begin{proof}
Suppose $d_{\hat{\Gamma}}(\hat{u}, \hat{v}) > 2$ for some $\hat{u}, \hat{v} \in \hat{\Gamma}$. Choose	 $u \in \pi^{-1}(\hat{u})$ and $v \in \pi^{-1}(\hat{v})$. It follows that $m(u,v) = \infty$. We will show that the bi-infinite geodesic $...uvuv...$ exhibits at least quadratic divergence. 

Let $\alpha$ be a shortest $B_{id}(2r)$--avoidant path from $(uv)^{r}$ to $(vu)^r$. Let $D$ be a Van-Kampen diagram with boundary path $(uv)^r\alpha (vu)^{-r}$. Write $(uv)^r = u_1v_1u_2v_2...u_rv_r$. Let $U_i$ denote the band emanating from $u_i$ and $V_i$ the band emanating from $v_i$. Note that for any $i, j$, $U_i$ cannot intersect $V_j$. For then, there would be an odd path in $\Gamma$ from $u$ to some $u'$, and an odd path from $v$ to some $v'$, so that $m(u',v') \ne \infty$. However, this would imply $d_{\hat{\Gamma}}(\hat{u}, \hat{v}) \le 1$, a contradiction. 

Fix $i$. Let $D_i$ denote the minimal connected subdiagram of $D$ containing both $U_i$ and $V_i$, and let $\alpha_i$ be the subsegment of $\alpha$ contained in $D_i$. Let $g = s_1...s_m$ be the word along the boundary path of $U_i$ from $u_i$ to $\alpha_i$. It follows $m \ge r - i$. Furthermore, $s_i \in A = \{ Star(t) ~ | ~ t \in \pi^{-1}(\hat{u}) \}$. Note that $d_{\hat{\Gamma}}(\hat{u}, \hat{t}) \le 1$ for $t \in A$. Let $S_j$ be the band in $D_i$ emanating from $s_j$. It follows $S_j$ cannot intersect $V_i$. For then $d_{\hat{\Gamma}}(\hat{u}, \hat{v}) \le 2$. Hence $S_j$ intersects $\alpha_i$ for each $j$. It follows, $|\alpha_i|$ is at least linear in $r - i$. Hence, $|\alpha|$ is at least quadratic in $r$.

\end{proof}

\begin{corollary} \label{coxeter_quadratic_corollary}
Let $W_{\Gamma}$ be an even Coxeter group. If $diam(\Gamma) > 2$, then $Div(W_{\Gamma})$ is at least quadratic. 
\end{corollary}

\bibliographystyle{amsalpha}
\bibliography{mybibliography}

\end{document}